\documentclass[english,11pt]{smfart}
\usepackage{etex}
\usepackage[english,french]{babel}
\usepackage[utf8]{inputenc}

\usepackage{amsbsy}
\usepackage{amsmath,amsfonts,amssymb,amsthm,mathrsfs,mathtools}
\usepackage[a4paper,vmargin={3.5cm,3.5cm},hmargin={2.5cm,2.5cm}]{geometry}
\usepackage[font=sf, labelfont={sf,bf}, margin=1cm]{caption}
\usepackage{graphicx}
\usepackage{epsfig}
\usepackage{latexsym}
\usepackage{xcolor}
\usepackage{ae,aecompl}
\usepackage{soul}
\usepackage{xcolor}
\usepackage[pdfpagemode=UseNone,bookmarksopen=false,colorlinks=true,urlcolor=blue,citecolor=blue,citebordercolor=blue,linkcolor=blue]{hyperref}
\usepackage{pstricks}
\usepackage{enumerate}
\usepackage{tikz,animate,media9}				
\usepackage{todonotes}
\usepackage[capitalize]{cleveref}
\usepackage{pifont}
\usepackage{bm,marvosym}
\usepackage{algorithm}
\usepackage{algorithmic}

\DeclareFontFamily{U}{BOONDOX-calo}{\skewchar\font=45 }
\DeclareFontShape{U}{BOONDOX-calo}{m}{n}{
  <-> s*[1.05] BOONDOX-r-calo}{}
\DeclareFontShape{U}{BOONDOX-calo}{b}{n}{
  <-> s*[1.05] BOONDOX-b-calo}{}
\DeclareMathAlphabet{\mathcalbis}{U}{BOONDOX-calo}{m}{n}
\SetMathAlphabet{\mathcalbis}{bold}{U}{BOONDOX-calo}{b}{n}
\DeclareMathAlphabet{\mathbcalboondox}{U}{BOONDOX-calo}{b}{n}

\makeatletter
\newcommand{\pushright}[1]{\ifmeasuring@#1\else\omit\hfill$\displaystyle#1$\fi\ignorespaces}
\newcommand{\pushleft}[1]{\ifmeasuring@#1\else\omit$\displaystyle#1$\hfill\fi\ignorespaces}
\makeatother

\usetikzlibrary{arrows, patterns, calc}	
\usepackage{bbm,enumerate,tabulary,booktabs}

\usepackage{enumitem}
	\renewcommand{\theenumi}{{(\roman{enumi})}}
	\renewcommand{\labelenumi}{\theenumi}
	
\newcommand{\Z}{\mathbb{Z}}
\renewcommand{\P}{\mathbb{P}}
\newcommand{\E}{\mathbb{E}}

\def\U{\mathbb{U}}
\def\N{\mathbb{N}}
\newcommand{\R}{\mathbb{R}}
\newcommand{\D}{\mathbb{D}}
\newcommand{\cut}{\mathsf{Cut}}
\def\S{\mathbb{S}}

\def\d{{\rm d}}

\renewcommand{\epsilon}{\varepsilon}
\newcommand\Es[1]{\mathbb{E}\left[#1\right]}
\renewcommand\Pr[1]{\mathbb{P}\left(#1\right)}

\newtheorem{theorem}{Theorem}[]

\newtheorem{proposition}[theorem]{Proposition}
\newtheorem{lemma}[theorem]{Lemma}
\newtheorem{corollary}[theorem]{Corollary}

\theoremstyle{definition}
\newtheorem{remark}[theorem]{Remark}

\def\llbracket{[\hspace{-.10em} [ }
\def\rrbracket{ ] \hspace{-.10em}]}

\def\build#1_#2^#3{\mathrel{
\mathop{\kern 0pt#1}\limits_{#2}^{#3}}}

\newcommand{\Yn}{Y^{(n)}}
\newcommand{\Zn}{Z^{(n)}}
\newcommand{\Wn}{W^{(n)}}

\newcommand{\cA}{\mathcal{A}}
\newcommand{\Tc}{\mathcal{T}}
\newcommand{\Ttc}{\tilde{\mathcal{T}}}
\newcommand{\Loop}{\mathsf{Loop}}
\newcommand{\Scoop}{\mathsf{Scoop}}
\newcommand{\Loopb}{\overline{\mathsf{Loop}}}
\newcommand{\q}{\mathsf{q}}
\newcommand{\m}{\mathbf{m}}
\newcommand{\Pq}{\mathbb{P}_{\q}}

\newcommand\BGW{\textup{\textrm{BGW}}}
\def \W {\mathsf{W}}
\def \H {\mathsf{H}}
\def \C {\mathsf{C}}
\def \Hc {\mathsf{H}^{\circ}}
\def \Mc {\mathcal{M}}

\def \Vc {\mathcal{V}}
\def \Vcn {\mathcal{V}_{n}}
\def \Tn {\mathcal{T}_{n}}
\def \Tgn {\mathcal{T}_{\geq n}}
\def \dc {d^{\circ}}

\def \Trunk {\mathsf{Trunk}}

\title[{\sffamily The boundary of  random planar maps}  {\scshape via}  {\sffamily looptrees}]{
{\sffamily The boundary of  random planar maps} \\  {\scshape via} \\ {\sffamily looptrees}}

\author{Igor Kortchemski}
\address{CNRS \& CMAP, \'Ecole polytechnique}
\email{igor.kortchemski@math.cnrs.fr}

\author{Loïc Richier}
\address{CMAP, \'Ecole polytechnique}
\email{loic.richier@polytechnique.edu}



\begin{document}

\maketitle

\begin{abstract}
We study the scaling limits of looptrees associated with Bienaymé--Galton--Watson (BGW) trees, that are obtained by replacing every vertex of the tree by a ``cycle'' whose size is its degree. First, we consider BGW trees whose offspring distribution is critical and in the domain of attraction of a Gaussian distribution. We prove that the Brownian CRT is the scaling limit of the associated looptrees, thereby confirming  a prediction of \cite{CK14b}.
Then, we deal with BGW trees whose offspring distribution is critical and heavy-tailed. We show that the scaling limit of the associated looptrees is a multiple of the unit circle. This corresponds to a so-called condensation phenomenon, meaning that the underlying tree exhibits a vertex with macroscopic degree. Here, we rely on an invariance principle for random walks with negative drift, which is of independent interest.
Finally, we apply these results to the study of the scaling limits of large faces of Boltzmann planar maps. We complete the results of \cite{Ric17} and establish a phase transition for the topology of these maps in the non-generic critical regime.
\end{abstract}

\begin{figure}[h!]
\centering
\begin{minipage}{.5\textwidth}
  \centering
  \includegraphics[width=\linewidth]{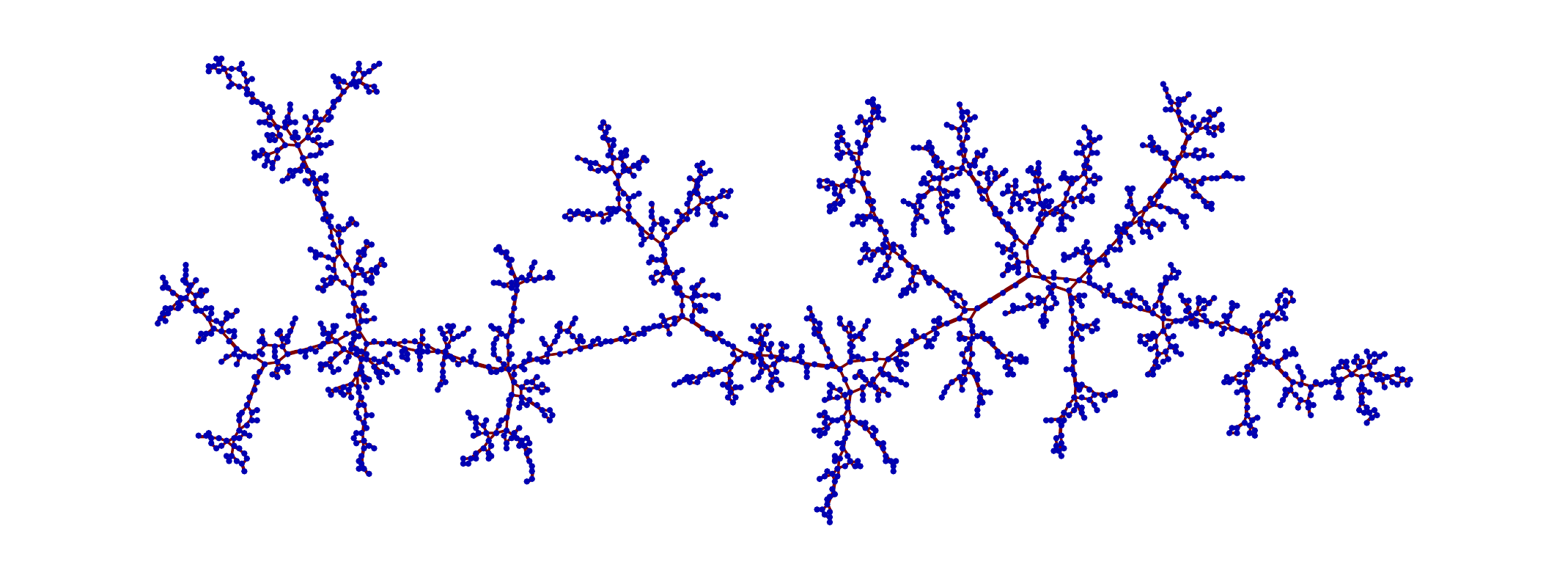}
  \label{fig:test1}
\end{minipage}\begin{minipage}{.5\textwidth} 
  \centering
  \includegraphics[width=\linewidth]{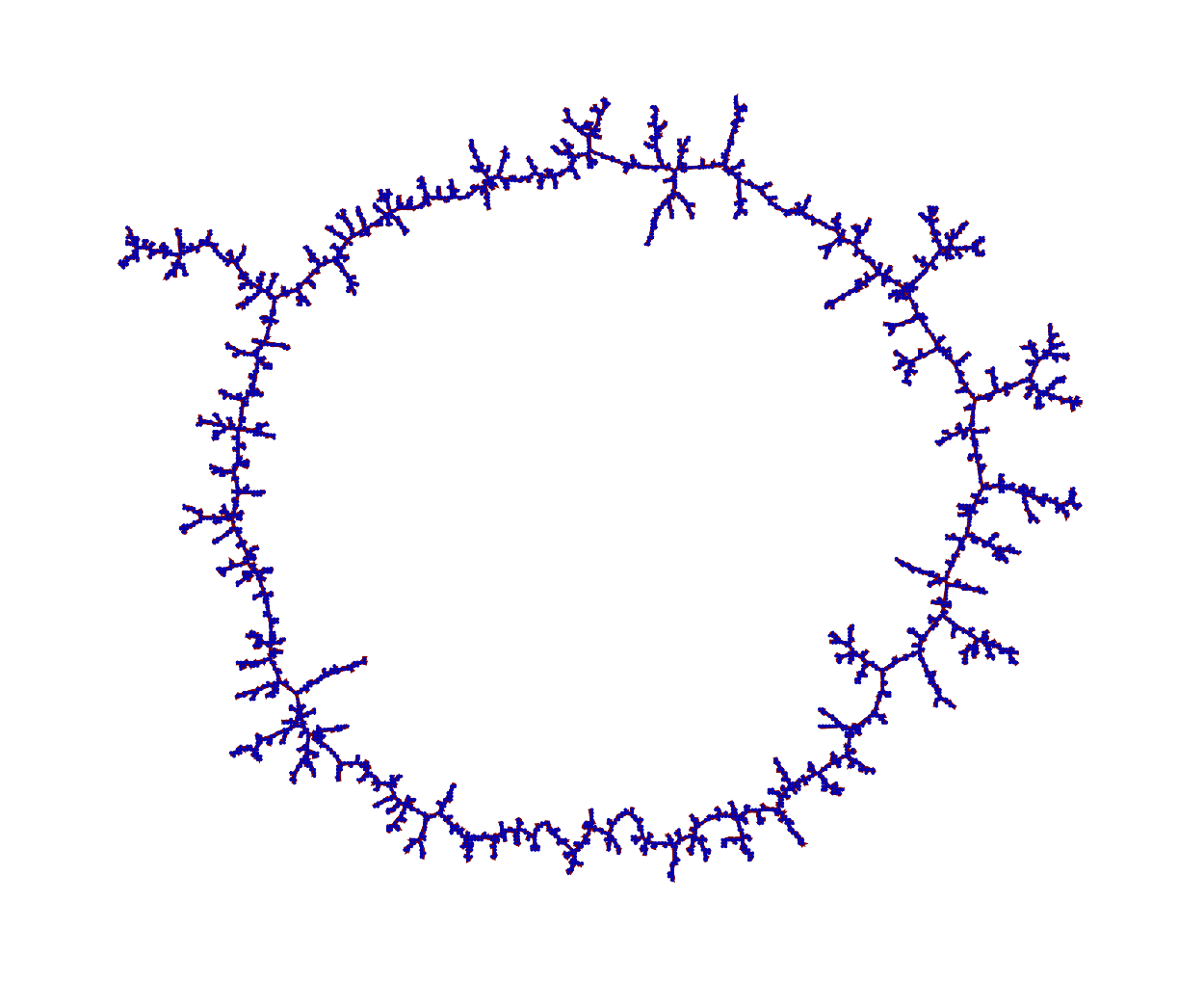}
  \label{fig:test2}
\end{minipage}
\caption{The looptree associated with a BGW tree whose offspring distribution is critical with finite variance (left) and subcritical with a heavy tail (right).}
\end{figure}

\section{Introduction}\label{sec:Intro}

\subsection{Scaling limits of random discrete looptrees}\label{sec:Looptreesdef} The purpose of this work is to study the scaling limits of discrete looptrees associated with large conditioned Bienaymé--Galton--Watson trees. By tree, we always mean \textit{plane tree}, that is a rooted ordered tree (with a distinguished corner and an ordering on vertices incident to each vertex). Given a probability measure $\mu$ on $\Z_{\geq 0}$, a Bienaymé--Galton--Watson tree with offspring distribution $\mu$ (in short, $\BGW_\mu$) is a random plane tree in which vertices have a number of offspring distributed according to $\mu$ all independently of each other (see Section \ref{sssec:trees} for more precise definitions).

Following \cite{CK14b}, with every plane tree $ \tau$ we associate a graph denoted by $\Loop(\tau)$ and called looptree. This graph has the same set of vertices as $\tau$, and for every vertices $u,v\in \tau$, there is an edge between $u$ and $v$ in $\Loop(\tau)$ if and only if $u$ and $v$ are consecutive children of the same parent in $\tau$, or if $v$ is the first or the last child of $u$ in $\tau$ (see Figure \ref{fig:loopintro} for an example). Roughly speaking, $\Loop(\tau)$ is obtained from $\tau$ by transforming vertices with degree $k$ into cycles of ``length'' $k$.

 \begin{figure}[!h]
 \begin{center}
 \includegraphics[width=  \linewidth]{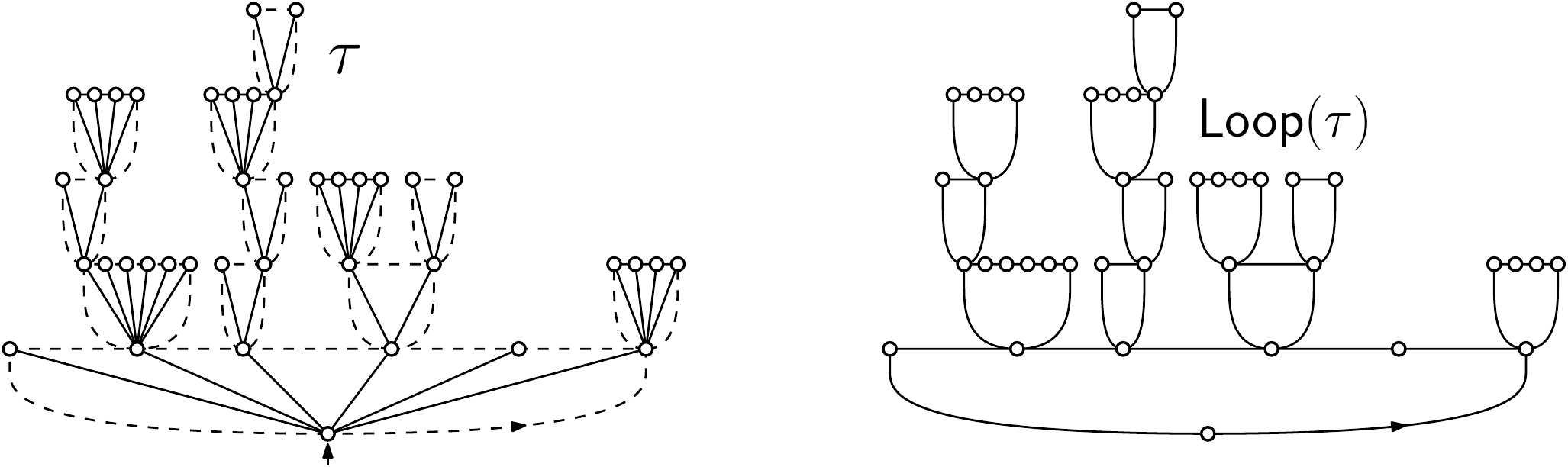}
 \caption{ \label{fig:loopintro}A plane tree $\tau$ and its associated looptree $ \mathsf{Loop}( \tau)$.}
 \end{center}
 \end{figure}

\smallskip

The study of scaling limits of discrete looptrees associated with $\BGW$ trees was initiated in \cite{CK14b}. The setting is the following: we let $\Tn$ be a $\BGW_\mu$ tree conditioned on having $n$ vertices, and aim at understanding the geometry of $\Loop(\Tn)$ when $n$ goes to infinity. More precisely, we view $\mathsf{Loop}( \Tn)$ as a compact metric space by endowing its vertices with the graph distance, and study the limit of rescaled versions of this metric space for the Gromov--Hausdorff topology. In the next part, for every $\lambda>0$ and every metric space $(E,d)$, the notation $\lambda\cdot E$ stands for $(E,\lambda \cdot d)$. We refer to \cite[Chapter 7.3]{burago_course_2001} for details on the Gromov--Hausdorff topology.

The work \cite{CK14b} deals with the case where $\mu$ is critical (i.e., has mean $m_\mu=1$) and falls within the domain of attraction of a stable law with parameter $\beta \in (1,2)$. The main result \cite[Theorem 4.1]{CK14b} shows that there exists a slowly varying function $L$ such that the convergence
	\begin{equation}\label{eqn:CVLoopCK}
		\frac{L(n)}{ n^{1/\beta}}\cdot \Loop(\Tn) \quad \xrightarrow[n\to\infty]{} \quad \mathscr{L}_{\beta},
	\end{equation}
	holds in distribution for the Gromov--Hausdorff topology, where $\mathscr{L}_{\beta}$ is the \textit{random stable looptree} with parameter $\beta\coloneqq(\alpha-1/2)^{-1}\in(1,2)$, also introduced in \cite{CK14b}. Recall that a function $L : \R_+ \rightarrow \R_+$ is slowly varying (at infinity) if for every $\lambda>0$ we have $L(\lambda x)/L(x) \rightarrow 1$ as $x \rightarrow \infty$ (see  \cite{BGT89} for more concerning slowly varying functions). 

This result was later completed  in \cite[Theorem 13]{CHK15}, in the case where $\mu$ is critical and has finite exponential moments. Then, there exists $C(\mu)>0$ such that the convergence \begin{eqnarray}\label{eqn:CVLoopCHK}  \frac{1}{\sqrt{n}}\cdot \mathsf{Loop}( \Tc_{n}) & \xrightarrow[n\to\infty]{} &    C(\mu) \cdot  \mathcal{T}_{ \mathbbm{e}}, \end{eqnarray}
holds in distribution for the Gromov--Hausdorff topology, where $\mathcal{T}_{ \mathbbm{e}}$ is Aldous' \textit{Brownian Continuum Random Tree} (CRT) coded by the normalized Brownian excursion $\mathbbm{e}$ \cite{Ald93} (see \cref{ssec:finalproof} for a formal definition).

Finally, although not stated in terms of looptrees, \cite{Kor15} treats the case where $\mu$ is subcritical (i.e., has mean $m_\mu<1$) and satisfies
\begin{equation}\label{eqn:HypLoc}
	\mu(i)= \frac{L(i)}{i^{\beta+1}},\quad i \in \N
\end{equation} for a certain $\beta>1$. Then the convergence
\begin{eqnarray}\label{eqn:CVLoopKor}  \frac{1}{n}\cdot \mathsf{Loop}( \Tc_{n}) & \xrightarrow[n\to\infty]{} &    (1-m_\mu) \cdot \mathbb{S}_{1},\end{eqnarray}
holds in distribution for the Gromov--Hausdorff topology, where $\mathbb{S}_1$ is the unit circle. This result stems from the existence of a \textit{condensation} phenomenon: when conditioned to be large, $\Tn$ exhibits a vertex with degree proportional to its total size~$n$. This was first observed in \cite{JS10,Jan12}, although the above result follows from the more precise analysis leaded in \cite{Kor15}.

\medskip

The contribution of this paper to the study of scaling limits of looptrees associated to $\BGW$ trees is twofold. In the first part (Theorem \ref{thm:circle}), we deal with subcritical offspring distributions that have a heavy tail, meaning that
\[\mu([i,\infty))=\frac{L(i)}{i^\beta}, \quad i\in\N\] for $L$ slowly varying and $\beta>1$. We emphasize that this condition is more general than the assumption \eqref{eqn:HypLoc} of \cite{Kor15}. However, this forces us to consider $\BGW$ trees conditioned to have \textit{at least} $n$ vertices.

In the second part (Theorem \ref{thm:CRT}) we improve the convergence \eqref{eqn:CVLoopCHK} established in \cite{CHK15} by considering critical offspring distributions $\mu$ falling within the domain of attraction of a Gaussian distribution, confirming thereby a prediction of \cite{CK14b}. 

The main motivation for these results comes from the study of scaling limits of large faces in random planar maps. As we will discuss at the end of this introduction, Theorems \ref{thm:circle} and \ref{thm:CRT} allow us to carry on the results of \cite{Ric17} dealing with the large scale geometry of these faces.

\paragraph*{Scaling limits of looptrees (circle regime).} Our first main result deals with looptrees associated to \textit{non-generic subcritical} $\BGW$ trees, meaning that the offspring distribution $\mu$ is subcritical and heavy-tailed.
 
\begin{theorem}
  \label{thm:circle}
   Let $\mu$ be a  offspring distribution with mean $m_\mu<1$. We assume that there exists $\beta>1$ and a slowly varying function $L$ such that, for every $i \geq 1$,
   \begin{equation}\label{eqn:HypTail}
   	\mu([i,\infty))= \frac{L(i)}{i^{\beta}}.
   \end{equation}
   Let also $J$ be the real-valued random variable such that $\Pr{J \geq x}=\left(\tfrac{1-m_\mu}{x}\right)^{\beta}$ for $x \geq 1-m_\mu$. Finally, for every $ n \geq 1$, let $ \Tgn$ be a $ \BGW_{ \mu}$ tree conditioned on having at least $n$ vertices. Then the convergence
\[ \frac{1}{n}\cdot \mathsf{Loop}( \Tgn)    \quad \xrightarrow[n\to\infty]{} \quad J \cdot \S_1\]
holds in distribution for the Gromov--Hausdorff topology, where $\S_1$ is the circle of unit length.
  \end{theorem}

This theorem roughly says that in the tree $ \Tgn$, for large $n$, there is a unique vertex with degree proportional to the total number of vertices. This phenomenon, known as \textit{condensation}, has already been observed under various forms in \cite{JS10,Jan12,Kor15}. One may hope to obtain a stronger result by considering $\BGW$ trees conditioned on having a fixed size, as in \eqref{eqn:CVLoopKor}. The rub is that without additional regularity assumptions on the offspring distribution $\mu$, it is not clear whether \cref{thm:circle} holds or not. The strategy of \cite{Kor15} is based on the so-called ``one big jump principle" of \cite{AL11}, that holds provided that $\mu$ is $(0,1]$\textit{-subexponential}. This means that if $X_1, \ldots, X_n$ are i.i.d.\ with law $\mu$, then for every $n\in\N$, \begin{equation}\label{eqn:DeltaSubExp}
	\Pr{S_n=x} \quad \underset{x\rightarrow \infty}{\sim} \quad n\Pr{X=x}.
\end{equation} However, there are offspring distributions satisfying \eqref{eqn:HypTail} but not \eqref{eqn:DeltaSubExp} (for instance, one can choose $\mu(2k)=k^{-\beta-1}$ and $\mu(2k+1)=\exp(-k)$ for $k\in\N$), and investigating whether \eqref{eqn:CVLoopKor} holds under a mere assumption on the tail distribution of $\mu$ is an interesting open question. Moreover, we do not know if the probability measure involved in our application to random planar maps satisfies \eqref{eqn:HypLoc} nor \eqref{eqn:DeltaSubExp} (see \cref{rem} in Section \ref{sec:Maps} for more on this). While our result is weaker than that of \cite{Kor15}, our assumptions are more general: roughly speaking, Theorem \ref{thm:circle} trades  the weaker conditioning for the stronger regularity assumption on~$\mu$.

   The proof of \cref{thm:circle} is based on an invariance principle for random walks with negative drift (\cref{thm:dTV}), which extends a result of Durrett \cite{Dur80} and  is of independent interest.
   
\paragraph*{Scaling limits of looptrees (CRT regime).} We now present our second main result, that deals with looptrees associated to $\BGW$ trees whose  offspring distribution $\mu$ is critical (i.e., has mean~$1$) and is in the domain of attraction of a Gaussian distribution. This means that the variance $\sigma_{\mu}^{2}$ of $\mu$ is either  finite, or there exists a slowly varying function $L$ such that $\mu([i,\infty))= {L(i)}/{i^{2}}$ for $i \geq 1$ (see \cite{IL71} for background on domains of attraction of stable laws).  

The scaling sequence $(B_{n} : n\geq 1)$ that will be involved in our limit theorem is defined as follows: if $X_{1}, X_{2}, \ldots$ are i.i.d.\ random variables with distribution $\mu$, then $(X_{1}+\cdots+X_{n}-n)/B_{n}$ converges in distribution to $\sqrt{2}$ times a standard Gaussian random variable (we use this normalization to keep the same definition of $B_{n}$ as in \cite{CK14b}). When $\sigma_{\mu}^{2}<\infty$, we may take $B_{n}=\sigma_{\mu} \sqrt{n/2}$, while when $\sigma_{\mu}^{2}=\infty$ there exists a slowly varying function $\ell$ such that $\ell(n) \rightarrow \infty$ and $B_{n}=\ell(n) \sqrt{n}$.

We finally, set \begin{equation}\label{eqn:CstSigma}
	c_\mu\coloneqq 
	\begin{cases}
	\quad \frac{1}{4}\left( \sigma_\mu^2+ 4- \mu(2 \Z+) \right) & \textrm{if } \sigma^{2}_{\mu}<\infty \\
	\quad \frac{1}{2} & \textrm{if } \sigma^{2}_{\mu}=\infty.
	\end{cases}
\end{equation} where $\mu(2 \Z+) =\mu(0)+\mu(2)+\cdots$.

\begin{theorem} \label{thm:CRT}
 Let $\mu$ be an offspring distribution with mean $m_\mu=1$ and in the domain of attraction of a Gaussian distribution. For every $ n \geq 1$, let $ \Tc_{n}$ be a $ \BGW_{ \mu}$ tree conditioned on having $n$ vertices.  Then the convergence
  \begin{eqnarray*}  \frac{1}{B_{n}}\cdot \mathsf{Loop}( \Tc_{n}) & \xrightarrow[n\to\infty]{} &    c_{\mu} \cdot  \sqrt{2} \mathcal{T}_{\mathbbm{e}} \end{eqnarray*}
  holds in distribution for the Gromov--Hausdorff topology, where $ \mathcal{T}_{ \mathbbm{e}}$ is the Brownian CRT.
  \end {theorem}

  The reason why  $\sqrt{2} \mathcal{T}_{   \mathbbm{e}}$  appears (instead of  $ \mathcal{T}_{\mathbbm{e}}$) simply comes from our normalisation convention for the scaling sequence $(B_{n})$.
  
  We refer to \cite{AM08,Bet15,CHK15,JS15,PSW16,Car16,CM16,stufler2016limits,Stu17} for a zoology of random discrete structures which are not trees, but whose scaling limits are $\mathcal{T}_{ \mathbbm{e}}$, the Brownian CRT.

As we have mentioned, when $\sigma_{\mu}^{2}<\infty$, the  result of \cref{thm:CRT} was already established in \cite[Theorem 13]{CHK15} under the existence of $\lambda>0$ such that $ \sum_{k \geq 0}\mu(k) e ^{\lambda k} < \infty$.
The improvement in \cref{thm:CRT} is important in three directions. First, the existence of small exponential moments does not hold \textit{a priori} in our application to random planar maps. Second, it is often challenging to relax an assumption involving a finite exponential moment condition to a finite variance condition: in particular, the proof of \cref{thm:CRT} uses different techniques than in \cite[Theorem 13]{CHK15}, and new ideas. We emphasize that until now, convergence towards the Brownian CRT of similar rescaled discrete weighted tree-like structures has mostly been obtained under finite exponential moment conditions (see \cite[Theorems 1, 13 and 14]{CHK15}, \cite[Theorem 5.1]{PSW16}, \cite[Theorem 6.60]{stufler2016limits}, and in particular the discussion in \cite[Section 3.3]{stefansson2017geometry}). Third, the method is robust, as it allows to treat the case $\sigma_{\mu}^{2}=\infty$, which was left as an open question in \cite{CK14b}.

The reason why the expression of $c_{\mu}$ depends on the finiteness of $\sigma_{\mu}^{2}$ is the following: when $\sigma_{\mu}^{2}<\infty$, the height of $\Tc_{n}$ and the typical sizes of loops of $\mathsf{Loop}( \Tc_{n})$ are of the same order $\sqrt{n}$, while when $\sigma_{\mu}^{2}=\infty$, the height of $\Tc_{n}$   (of order $ \tfrac{n}{B_{n}}$) is negligible compared to the typical size of loops in  $\mathsf{Loop}( \Tc_{n})$ (of order $B_{n}$), so that asymptotically distances in $\Tc_{n}$ do not contribute to distances in $ \mathsf{Loop}( \Tc_{n})$, in  contrast with the finite variance case.

Note that the classification of the scaling limit of looptrees associated with conditioned critical $\BGW$ trees whose offspring distribution is in the domain of attraction of a stable distribution of index $\alpha \in (1,2]$ is now complete: \cite[Theorem 4.1]{CK14b} covers the case $\alpha \in (1,2)$, Theorem~\ref{thm:CRT} covers the case $\alpha=2$ (both with finite variance and infinite variance).

\subsection{A spinal decomposition.}
\label{ss:spinal}

The proof of \cref{thm:CRT} relies on a spinal decomposition which is interesting in its own, and that we now detail. We refer to Section \ref{sec:Trees} for definitions concerning plane trees. First, let us introduce some notation. If $\tau$ is a tree and $u$ a vertex of $\tau$, we denote by $\Trunk(\tau,u)$ the tree made by vertices that are ancestors of $u$ in $\tau$, together with their neighbours (see Figure \ref{fig:Trunk} and Section \ref{sec:Th2} for details). If $\tau$ is a tree, we also denote by $\Lambda(\tau)$ its number of leaves (that is, childless vertices).

If $\mu$ is an offspring distribution with mean $1$, we let $\mu^{\ast}$ be the size-biased version of $\mu$ defined~by \[\mu^{\ast}(j)=j\mu(j), \quad j\geq 0.\] Then, denote by $\Trunk^{\ast}_{h}$ the ``size-biased trunk'' of height $h$ defined as follows: it is a tree made of a \textit{spine} with vertices $v^*_0$,$v^*_1$,\ldots,$v^{\ast}_{h-1}$ having independent number of children distributed according to $\mu^{\ast}$. For every $0 \leq i \leq h-1$, among all children of $v^{\ast}_{i}$, the child belonging to the spine is uniform, while its other children are leaves. Also, $v^{\ast}_{h}$ is a leaf. 

Note that $\Trunk^{\ast}_{h}$ may be seen as part of the spine of the infinite BGW tree conditioned to survive, that was first defined in \cite{Kes86}.

 \begin{figure}[!h]
 \begin{center}
 \includegraphics[width= .5 \linewidth]{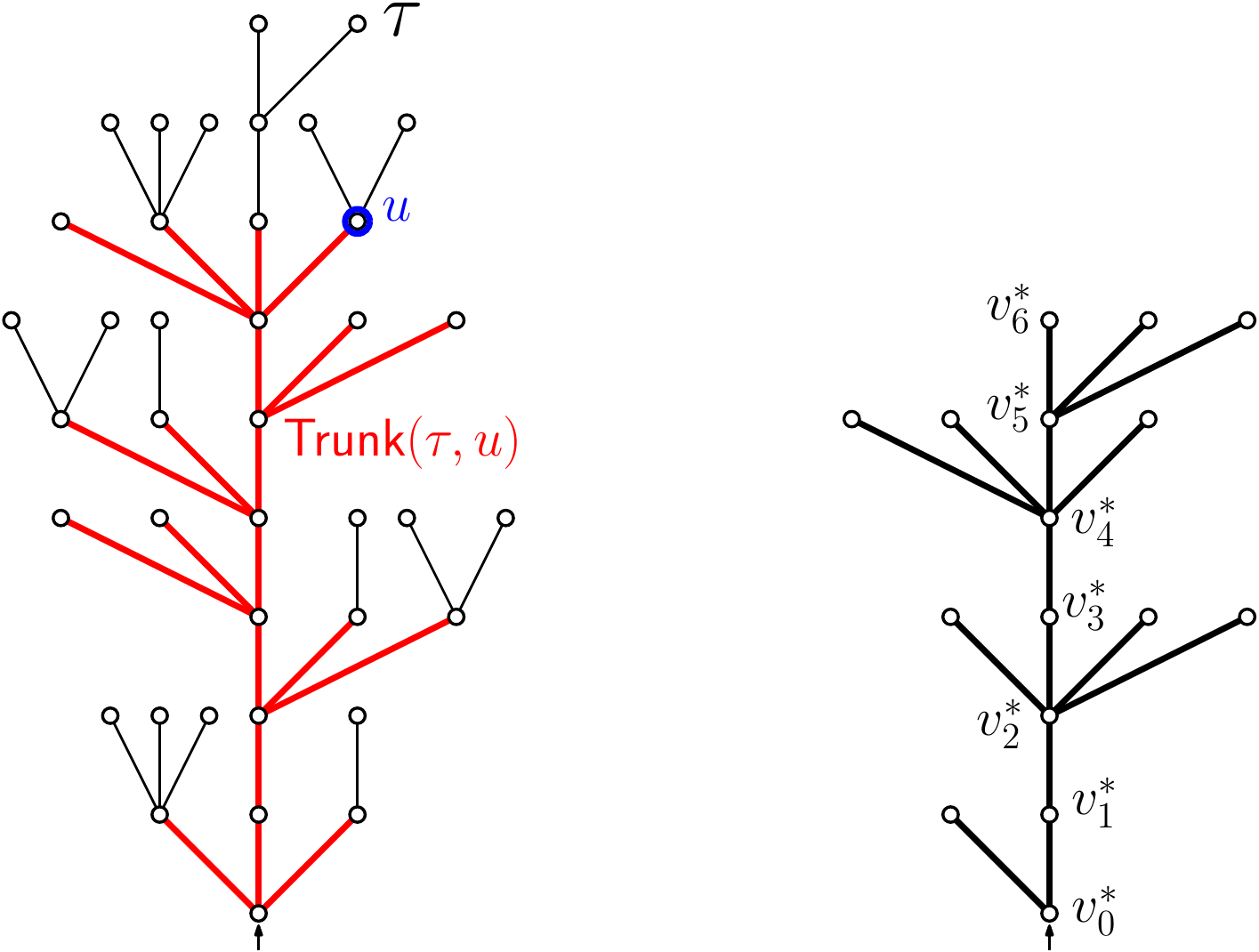}
 \caption{ \label{fig:Trunk}Left: A plane tree $\tau$ with a marked vertex $u$ and the associated ``trunk" $\Trunk(\tau,u)$ (in bold red). Right: a realization of $\Trunk^{\ast}_{6}$, with $\Lambda(\Trunk^{\ast}_{6})=10$ leaves.}
 \end{center}
 \end{figure}

We refer to \cite{Lin92} or \cite[Section~2]{dHo12} for background concerning the total variation distance, that we denote by $d_{\mathrm{TV}}$.

 \begin{theorem}
 \label{thm:trunk2}
 Let $\mu$ be an  offspring distribution with mean $m_\mu=1$ that is in the domain of attraction of a Gaussian distribution. For every $n\geq 1$, let $\Tc_n$ be a $\BGW_\mu$ tree conditioned to have $n$ vertices. 
 \begin{enumerate}
 \item[(i)] Fix $t>0$ and let $ \Vc_{n} ^{t}$ be a vertex chosen uniformly at random in $\Tc_n$ among all those at height $\lfloor t \frac{n}{B_{n}} \rfloor$. Then  \[d_{\textup{TV}}\left( \Trunk( \mathcal{T}_{n},\Vc_{n} ^{t}), \Trunk^{\ast}_{\lfloor t  \frac{n}{B_{n}} \rfloor }\right)  \quad \xrightarrow[n \rightarrow \infty]{} \quad 0.\]
 \item[(ii)]  Let $ \Vc_{n}$ be a vertex chosen uniformly at random in $ \mathcal{T}_{n}$, and $ \mathcal{R}$ be a random variable with density $2x e^{-x^{2}} \mathbbm{1}_{x \geq 0} {\d}x$. Then \[d_{\textup{TV}}\left( \Trunk( \mathcal{T}_{n}, \Vc_{n} ),\Trunk^{\ast}_{\lfloor  \mathcal{R}  \frac{n}{B_{n}} \rfloor }\right)  \quad \xrightarrow[n \rightarrow \infty]{} \quad 0.\]
 \end{enumerate}
  \end{theorem}
 
This is consistent with the well-known fact that when renormalized by $B_{n}/n$, the height $|\Vc_{n}|$ of $\Vc_n$ converges toward $\mathcal{R}$ in distribution. We will use \cref{thm:trunk2} to deduce asymptotic properties of $ \Trunk( \mathcal{T}_{n}, \Vc_{n} )$  from those of $\Trunk^{\ast}$, which is much simpler to study.

\subsection{Applications to random planar maps.} The main motivation of this work is the study of large faces in Boltzmann planar maps.

The Boltzmann measures are defined out of a weight sequence $\q=(q_1,q_2,\ldots)$ of nonnegative real numbers assigned to the faces of the maps. Precisely, the Boltzmann weight of a bipartite planar map $\m$ (that is, with faces of even degree) is given by
\[w_\q(\m)\coloneqq\prod_{f \in \textup{Faces}(\m)}q_{\deg(f)/2}.\] The sequence $\q$ is called \textit{admissible} when these weights form a finite measure on the set of rooted bipartite maps (i.e.\ with a distinguished oriented edge called the \textit{root edge}). The resulting probability measure $\Pq$ is the Boltzmann measure with weight sequence $\q$. We say that $\q$ is critical if the expected squared number of vertices of a map is infinite under $\Pq$, and subcritical otherwise (see Section~\ref{sec:ssmaps} for precise definitions).

\smallskip

The scaling limits of Boltzmann bipartite maps conditioned to have a large number of faces have been actively studied. In 2013, Le Gall \cite{LG11} and Miermont \cite{Mie11} proved the convergence of uniform quadrangulations towards the so-called \textit{Brownian map}. This result has been extended to critical sequences $\q$ such that the degree of a typical face has finite variance in \cite{Mar16} (we then say that $\q$ is \textit{generic critical}) building on the earlier works \cite{MM07,LG11}.
A different scaling limit appears when we assume that the critical sequence $\q$ is such that the degree of a typical face is in the domain of attraction of a stable law with parameter $\alpha\in(1,2)$ (we then say that $\q$ is \textit{non-generic critical} with parameter $\alpha$).
Under slightly stronger assumptions, Le Gall and Miermont \cite{LGM09} proved the subsequential convergence of such Boltzmann maps. There is a natural candidate for the limiting compact metric space, called the \textit{stable map} with parameter $\alpha$.
The geometry of the stable maps is dictated by large faces that remain present in the scaling limit. The behaviour of these faces is believed to differ in the dense phase $\alpha\in(1,3/2)$, where they are supposed to be self-intersecting, and in the dilute phase $\alpha\in(3/2,2)$, where it is conjectured that they are self-avoiding. Our work is a first step towards this dichotomy.

\medskip

The strategy initiated in \cite{Ric17} consists in studying Boltzmann maps \textit{with a boundary}. This means that we view the face on the right of the root edge as the boundary $\partial\m$ of the map $\m$. Consequently, this face receives no weight, and its degree is called the \textit{perimeter} of the map. Then, for every $k\geq 0$, we let $\Mc_k$ be a Boltzmann map with weight $\q$ conditioned to have perimeter $2k$. The boundary $\partial\Mc_k$ of this map can be thought of as a typical face of degree $2k$ of a Boltzmann map with weight $\q$. The main result of \cite{Ric17} deals with the dense regime. It shows that if $\q$ is a non-generic critical weight sequence with parameter $\alpha\in(1,3/2)$, there exists a slowly varying function $L$ such that in distribution for the Gromov--Hausdorff topology,
	\[\frac{L(k)}{ (2k)^{\alpha-1/2}}\cdot \partial \Mc_k \quad \xrightarrow[k\to\infty]{} \quad \mathscr{L}_{\beta}, \]  where $\mathscr{L}_{\beta}$ is the random stable looptree with parameter $\beta\coloneqq(\alpha-1/2)^{-1}\in(1,2)$.
	
\smallskip

The purpose of this work is to investigate the subcritical, dilute and generic critical regimes that were left untouched in \cite{Ric17}. Thanks to the results of \cite{Ric17}, this problem boils down to the study of scaling limits of \textit{discrete looptrees}, in the specific regimes that we dealt with in Theorems \ref{thm:circle} and \ref{thm:CRT}. Let us now state the applications of these results to the scaling limits of the boundary of Boltzmann planar maps. We start with the dilute and generic critical regimes. 

\begin{corollary}
\label{cor:ScalingDilute}
Let $\q$ be a critical weight sequence which is either generic, or non-generic with parameter $\alpha\in(3/2,2)$ (dilute phase). For every $k\geq 0$, let $\Mc_{\geq k}$ be a Boltzmann map with weight sequence $\q$ conditioned to have perimeter at least $2k$. Then, there exists a non-degenerate random variable $J_\q$ such that the convergence
	\[\frac{1}{2k}\cdot \partial \Mc_{\geq k} \quad \xrightarrow[k\to\infty]{} \quad J_\q \cdot \mathbb{S}_1\] holds in distribution for the Gromov--Hausdorff topology.
\end{corollary}

This result is consistent with the conjecture that large faces are self-avoiding in the dilute phase. It is also related to \cite[Theorem 8]{BM17}, which states that some generic critical Boltzmann maps (called \textit{regular}) conditioned to have large perimeter converge towards the so-called \textit{free Brownian disk}, that has the topology of the unit disk. Of course, the reason why we need to consider maps having perimeter \emph{at least} $2k$ is the same as in Theorem \ref{thm:circle} (see Remark \ref{rem} in Section \ref{sec:Maps} for details).

\smallskip

We now deal with the subcritical regime. Intuitively, when conditioning a subcritical Boltzmann map to have a large face, this face ``folds" onto itself, forcing it to become tree-like.

\begin{corollary}
\label{cor:ScalingSubcritical}Let $\q$ be a subcritical weight sequence. For every $k\geq 0$, let $\Mc_{k}$ be a Boltzmann map with weight sequence $\q$ conditioned to have perimeter $2k$. Then, there exists a constant $K_\q>0$ such that the convergence	\[\frac{1}{\sqrt{2k}}\cdot \partial \Mc_k \quad \xrightarrow[k\to\infty]{} \quad K_\q \cdot \mathcal{T}_\mathbbm{e} \]
holds in distribution for the  Gromov--Hausdorff topology. 	
\end{corollary}

This result should be compared with \cite[Theorem 1.1]{JS15} (see also \cite[Theorem 5]{Bet15}), where the convergence of subcritical Boltzmann maps conditioned to have large volume towards the Brownian CRT is proved. On the one hand, \cite{JS15} deals with the whole map (not only its boundary), but on the other hand, our assumptions are more general (in \cite{JS15}, it is assumed beyond subcriticality that typical faces have a heavy-tailed distribution in a quite strong sense).

Together with the results of \cite{Ric17}, \cref{cor:ScalingDilute,cor:ScalingSubcritical} give a global picture of the scaling limits of the boundary of Boltzmann maps (see Figure \ref{fig:Array} for an illustration). In particular, Corollary \ref{cor:ScalingDilute} together with \cite[Theorem 1.1]{Ric17} establish the phase transition in the parameter $\alpha$ for the topology of large faces in Boltzmann maps, that was only overviewed through local limits in \cite[Theorem 1.2]{Ric17}, and \textit{via} volume growth exponents in \cite{BC17}. Note however that establishing a result similar to Corollary \ref{cor:ScalingDilute} for the boundary $\partial\Mc_k$ of a map conditioned to have perimeter \textit{exactly} $2k$ is still an open problem. Finally, our results also have an interpretation in terms of large loops in the rigid $\mathcal{O}(n)$ loop model on quadrangulations, by applying the results of \cite{BBG12} (see \cite{Ric17} for more on this). 

 \begin{figure}[!h]
 \begin{center}
 \includegraphics[width=1   \linewidth]{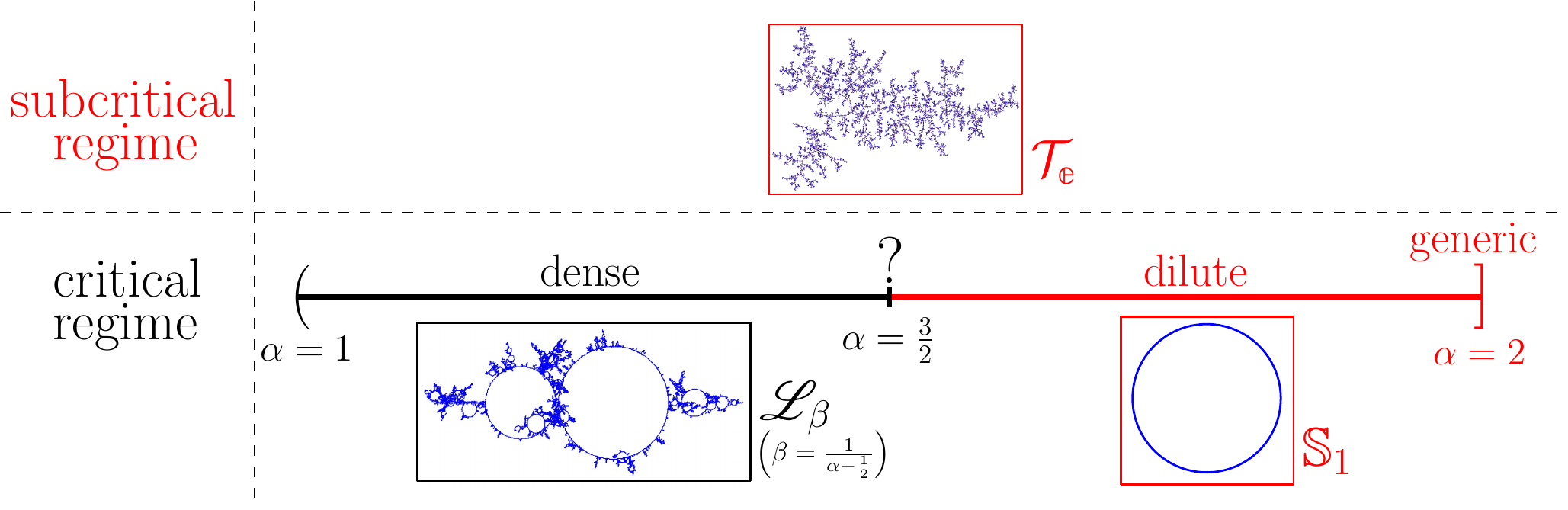}
 \caption{ \label{fig:Array} A summary of the scaling limits of the boundary of Boltzmann planar maps. The contributions of this paper are indicated in red. The generic critical weight sequences are usually identified with the parameter $\alpha=2$, because the distribution of a typical face has finite variance and thus belongs to the domain of attraction of a Gaussian distribution.}
 \end{center}
 \end{figure}

\begin{remark}[Special case $\alpha=3/2$]   As proved in \cite{Ric17}, the $\BGW$ tree structure describing the boundary of non-generic critical Boltzmann maps with parameter $3/2$ can be either subcritical, or critical. As suggested in \cite[Remark 6.3]{Ric17}, we expect the scaling limit to be a multiple of a loop in both cases, although the condensation phenomenon could occur at a scale smaller than the total number of vertices in the critical tree setting. This has been investigated in \cite{KR18}.
\end{remark}

\medskip

The paper is organized as follows. In Section \ref{sec:Trees}, we recall definitions and fundamental results about (random) plane trees. Then, Sections \ref{sec:Th1} and \ref{sec:Th2} are devoted to the proofs of Theorem \ref{thm:circle} and Theorem \ref{thm:CRT}. The first is based on a limit theorem for random walks with negative drift, while the others use a spinal decomposition and a tightness argument. Finally, we discuss the applications to random planar maps in Section \ref{sec:Maps}.

\subsection*{Acknowledgements.} Many thanks to Nicolas Curien, Sébastien Martineau, Cyril Marzouk and Grégory Miermont for comments on a preliminary version of this work. We are also grateful to Cyril Marzouk for stimulating discussions concerning Lemma \ref{lem:tightfd} and to a referee for suggestions that improved the paper.  We acknowledge  partial support from  Agence Nationale de la Recherche, grant number ANR-14-CE25-0014 (ANR GRAAL). I.K. acknowledges  partial support from the City of Paris,  grant “Emergences Paris 2013, Combinatoire à Paris”.

\tableofcontents

\section{Trees}\label{sec:Trees}

\subsection{Plane trees.}
\label{sssec:trees}
We use Neveu's formalism \cite{Nev86} to define plane trees: let $\N = \{1, 2, \dots\}$ be the set of all positive integers, set $\N^0 = \{\varnothing\}$ and consider the set of labels $\U = \bigcup_{n \ge 0} \N^n$. For $v = (v_1, \dots, v_n) \in \U$, we denote by $|v| = n$ the length of $v$. If $n \ge 1$, we define $\textup{pr}(v) = (v_1, \dots, v_{n-1})$ and for $i \ge 1$, we let $vi = (v_1, \dots, v_n, i)$. More generally, for $w = (w_1, \dots, w_m) \in \U$, we let $vw = (v_1, \dots, v_n, w_1, \dots, w_m) \in \U$ be the concatenation of $v$ and $w$. We endow $\U$ with the lexicographical order (denoted by $\prec$): given $v,w \in \U$, if $z \in \U$ is their longest common prefix (so that $v = z(v_1, \dots, v_n)$ and $w = z(w_1, \dots, w_m)$ with $v_1 \ne w_1$), then we have $v \prec w$ if $v_1 < w_1$.

A (locally finite) \emph{plane tree} is a nonempty  subset $\tau \subset \U$ such that (i) $\varnothing \in \tau$; (ii) if $v \in \tau$ with $|v| \ge 1$, then $\textup{pr}(v) \in \tau$; (iii)  if $v \in \tau$, then there exists an integer $k_v(\tau) \ge 0$ such that $vi \in \tau$ if and only if $1 \le i \le k_v(\tau)$.

We may view each vertex $v$ of a tree $\tau$ as an individual of a population for which $\tau$ is the genealogical tree. For $v,w \in \tau$, we let  $\llbracket v, w \rrbracket$ be the vertices belonging to the shortest path from $v$ to $w$ in $\tau$. Accordingly, we use $\llbracket v, w \llbracket$ for the same set, excluding $w$. The vertex $\varnothing$ is called the \emph{root} of the tree and for every $v \in \tau$, $k_v(\tau)$ is the number of children of $v$ (if $k_v(\tau) = 0$, then $v$ is called a \emph{leaf}), $|v|$ is its \emph{generation}, $\textup{pr}(v)$ is its \emph{parent} and more generally, the vertices $v, \textup{pr}(v), \textup{pr} \circ \textup{pr} (v), \dots, \textup{pr}^{|v|}(v) = \varnothing$ belonging to $ \llbracket \varnothing, v \rrbracket$ are its \emph{ancestors}. If $\tau$ is a tree and $v$ a vertex of $\tau$, $\theta_{v}(\tau)=  \{vw \in\tau : w\in\mathbb{U}\}$ denotes the tree made of $v$ together with its descendants in $\tau$. We also let $\cut_{v}(\tau)= \{v\} \cup \tau \backslash \theta_{v}(\tau)$ be tree obtained from $\tau$ by removing all the (strict) descendants of $v$ in $\tau$. Finally, we let $|\tau|$ be the total number of vertices (or size) of the plane tree $\tau$. For every  $n \geq 1$, we let $\mathbb{A}_n$ be the set of plane trees with $n$ vertices.

\subsection{Bienaymé--Galton--Watson trees and their codings.}
\label{sssec:coding}
Let $\mu$ be a  probability measure on $\Z_{\geq 0}$ (called the \emph{offspring distribution}) such that $\mu(0) > 0$ and $\mu(0)+\mu(1)<1$ (to avoid trivial cases). We also assume that $\mu$ has mean $m_{\mu}:=\sum_{i \geq 0} i \mu (i) \leq 1$. The Bienaymé--Galton--Watson ($\BGW$) measure with offspring distribution $\mu$ is the probability measure $\BGW_\mu$ characterized by
\begin{equation}\label{eq:def_GW}
\BGW_\mu(\tau) = \prod_{u \in \tau} \mu(k_u(\tau)),
\end{equation}
for every plane tree $\tau$, see e.g.~\cite[Prop.~1.4]{LG05}. We say that the offspring distribution $\mu$ (or a tree with law $\BGW_\mu$) is \textit{critical} (resp.\ \textit{subcritical}) if $m_\mu=1$ (resp.\ $m_\mu<1$). For the sake of simplicity, we always assume that $\mu$ is \textit{aperiodic}, meaning that $\textup{gcd}(\{k\geq 1 : \mu(k)>0\})=1$. However, all the results can be adapted to the periodic setting without effort. Also, when dealing with $\BGW_\mu$ tree conditioned to have $n$ vertices, we always implicitly assume that we work along a subsequence on which $\BGW_\mu(\mathbb{A}_n)>0$.

Consider a tree $\tau$ with its vertices listed in lexicographical order:
$\varnothing=u_{0}\prec u_{1}\prec \cdots \prec u_{|\tau|-1}$. The \textit{height function} 
$\H(\tau)=(\H_n(\tau) : 0 \leq n < |\tau|)$ is defined, for $0
\leq n < |\tau|$, by $\H_n(\tau)=|u_n|$. The \textit{{\L}ukasiewicz path} $ \W(\tau)=( \W_n(\tau) : 0 \leq n < |\tau|)$ of a tree $\tau$ is defined by $ \W_0(\tau)=0$, and $ \W_{n+1}(\tau)= \W_{n}(\tau)+k_{u_n}(\tau)-1$
for $0 \leq n < |\tau|$.

Finally, the \textit{contour function} $\C(\tau)=( \C_n(\tau) : 0
\leq n \leq 2(|\tau|-1) )$ of a tree $\tau$ is defined by considering a particle that starts from the root and visits continuously all edges at unit speed (assuming
that every edge has unit length), going backwards as little as possible and respecting the lexicographical order of vertices. If we let $\varnothing= x_0, \ldots, x_{2(|\tau |-1)}$ be the ordered list of vertices of $\tau$ visited by the particle (with repetition), then we have $\C_n(\tau)=|x_n|$ for $0
\leq n \leq 2(|\tau|-1)$ (so that $\C_t(\tau)$ is the distance to
the root of the position of the particle at time $t$, see \cite[Section 2]{Du03} for more on this).

For technical reasons, we let $\H_{n}(\tau)=\W_n(\tau)=0$ for $ n> |\tau|$ and $\C_n(\tau)=0$ for $n > 2(|\tau|-1)$. We also extend $\H(\tau)$, $\W(\tau)$ and $\C(\tau)$ to $\R_+$ by linear interpolation.

\section{Looptrees: the non-generic subcritical case}\label{sec:Th1}

\subsection{Invariance principle for random walks with negative drift}

The roadmap to \cref{thm:circle} is based on a limit theorem for random walks with negative drift that we now state.  Let $(X_{i} : i \geq 1)$ be an i.i.d.\ sequence of real-valued random variables such that:
\begin{enumerate}
\item[--] $\Es{X_{1}}=-\gamma<0$ 
\item[--] $\Pr{X_{1} \geq x}=L(x)x^{-\beta}$ with $\beta > 1$ and $L$ a slowly varying function at infinity.
\end{enumerate}

We set $W_{0}=0$, $W_{n}=X_{1}+\cdots+X_{n}$ for every $n \geq 1$ and let \[\zeta= \inf \{i \geq 1 : W_{i}<0\}.\] We also set $W_{i}=0$ for $i<0$ by convention. In this section, our goal is to study the behaviour of the random walk $(W^{(n)}_{i} : i \geq 0)$ under the conditional probability   $ \P( \, \cdot \, | \zeta \geq n)$, as $n \rightarrow \infty$. More precisely, we shall couple with high probability the trajectory $(W^{(n)}_{i} : i \geq 0)$ with that of a random walk conditioned to be nonnegative for a random number of steps, followed by an independent ``big jump'', and then followed by an independent unconditioned random walk.

\paragraph*{Statement of the main result}

For $n \geq 1$, we consider the process $(\Zn_{i} : i \geq 0)$ whose distribution is specified as follows. First, let $I$ be a random variable with law given by
\[\Pr{I=j}= \frac{\Pr{\zeta\geq j}}{\Es{\zeta}}, \qquad j \geq 1.\] Note that $\Es{\zeta}<\infty$ since $W_{n}$ has negative drift (see e.g.~\cite[Theorem 1, Sec.~XII.2]{Fel71}).
Then, for every $j \geq 1$, conditionally given $\{I=j\}$, the three random variables $(\Zn_{i} : 0\leq i <j)$, $\Yn_{j} \coloneqq \Zn_j-\Zn_{j-1}$ and $(\Zn_{i+j}-\Zn_j : i \geq 0)$ are independent and distributed as follows:
\begin{itemize}
	\item $\displaystyle (\Zn_{i} : 0\leq i <j) \, \mathop{=}^{(d)}  \, (W_{i} : 0\leq i <j) \textrm{ under } \P(\, \cdot \, | \zeta \geq j)$
	\item $\displaystyle \Yn_{j}  \,  \mathop{=}^{(d)} \, X_{1} \textrm{ under } \P(\, \cdot \, | X_{1} \geq \gamma n)$
	\item $\displaystyle (\Zn_{i+j}-\Zn_j : i \geq 0) \, \mathop{=}^{(d)}   \, (W_{i} : i \geq 0)$.
\end{itemize}

\begin{theorem}\label{thm:dTV}
Let $(\Wn_{i} : i \geq 0)$ be distributed as the random walk $(W_{i} : i \geq 0)$ under the conditional probability $ \P( \, \cdot \, | \zeta \geq n)$. Then, we have
\[d_{\mathrm{TV}}\left( \left(\Wn_{i} : i \geq 0 \right),  \left(\Zn_{i} : i \geq 0 \right) \right)  \quad \xrightarrow[n\to\infty]{}  \quad 0,\]
where $d_{\mathrm{TV}}$ denotes the total variation distance on $\R^{\Z_{+}}$ equipped with the product topology.
\end{theorem}

Intuitively speaking, this means that under the conditional probability $ \P( \, \cdot \, | \zeta \geq n)$, as $n \rightarrow \infty$, the random walk $(W_{i} : i \geq 0)$ first behaves as conditioned to stay nonnegative for a random number $I$ of steps, then makes a jump distributed as  $ \P( \, \cdot \, | X_1 \geq \gamma n)$, and finally evolves as a non-conditioned walk. 

\subsection{Proof of \cref{thm:dTV}}

In order to establish \cref{thm:dTV}, our main input is a result describing the asymptotic behavior of the tail distribution $\Pr{\zeta \geq n}$ as $n \rightarrow \infty$:
\begin{equation}
\label{eq:BB}\Pr{\zeta \geq n}  \quad \mathop{\sim}_{n \rightarrow \infty} \quad \Es{\zeta} \Pr{X_{1} \geq \gamma n}.
\end{equation}
This follows from \cite[Theorem 8.2.4]{BB08}. Indeed, in the notation of \cite{BB08}, $X_{1}$ belongs to the class $ \mathcal{R}$  of distributions with regularly varying right tails, see \cite[Equation 8.2.3]{BB08}. See also \cite[Theorem 3.2, Remark 3.3]{DSV13}.

Let us first introduce some notation. We denote by $\cA$ be the Borel $\sigma$-algebra on $\R^\N$ associated with the product topology, and we set
\[\mu_n(A) \coloneqq \Pr{\left(\Wn_i : i \geq 0\right)\in A} \quad \text{and} \quad \nu_n(A) \coloneqq \Pr{\left(\Zn_i : i \geq 0\right)\in A}, \quad A\in \cA.\] 
The idea of the proof of \cref{thm:dTV} is to transform the estimate \eqref{eq:BB} into an estimate on probability measures by first finding a ``good" event $G_n$ such that $\nu_{n}(G_{n}) \rightarrow 1$ as $n \rightarrow \infty$ and then by showing that $\sup_{A \in \cA }{\vert \mu_n(A\cap G_{n})-\nu_n(A\cap G_{n})\vert} \rightarrow 0$ as $n \rightarrow\infty$.

\begin{lemma}\label{lem:PF} For every $n\in\N$ and $\gamma>0$, set 
\[
	G_{n}:=\left\lbrace (w_0,\ldots,w_n) \in \R_+^{n+1} : \exists! \ i \in \llbracket 1,n\rrbracket \text{ s.t. } w_i-w_{i-1} \geq \gamma n \right\rbrace.
\] 
Then $\nu_{n}( G_{n} ) \longrightarrow 1$ as $n \rightarrow \infty$.
\end{lemma}

Let us first explain how one establishes  \cref{thm:dTV} using \cref{lem:PF}.

\begin{proof}[Proof of \cref{thm:dTV}]By Lemma \ref{lem:PF}, it suffices to see that $\sup_{A \in \cA }{\vert \mu_n(A\cap G_n)-\nu_n(A\cap G_n)\vert} \rightarrow 0$ holds as $n \rightarrow\infty$. Without loss of generality, we  focus on events of the form $\mathbf{w}\times A$, where $\mathbf{w}=(0,w_1,\ldots,w_n)\in G_n$ and $A\in \cA$. On the one hand, since $\mathbf{w} \in G_{n}$ we have
\[\mu_n\left(\mathbf{w}\times A\right)=\frac{\Pr{\left(W_0,\ldots,W_n \right)=\mathbf{w}} \Pr{\left(W_{n+i} : i \geq 1\right) \in A }}{\Pr{\zeta \geq n}}.\] On the other hand,  write
\[\nu_n\left(\mathbf{w}\times A\right)=\sum_{j=1}^\infty \Pr{I=j, \ \left(\Zn_i : i \geq 0\right) \in \mathbf{w}\times A}.\] Since $\mathbf{w}\in G_n$, we get
\begin{align*}
& \Pr{I\leq n, \ \left(\Zn_i : i \geq 0\right) \in \mathbf{w}\times A}\\
	&=\sum_{j=1}^n \Pr{I=j}\cdot  \frac{\Pr{\left(W_0,\ldots,W_{j-1} \right)=(w_0,\ldots,w_{j-1})}}{\Pr{\zeta \geq j}} \cdot \frac{\Pr{X_1=w_j-w_{j-1}, \ X_1\geq \gamma n }}{\Pr{X_1 \geq \gamma n }}  \\
&	\hspace{2cm}\qquad\qquad\qquad\qquad\qquad\qquad\qquad\cdot \Pr{\left(W_{i+j} : i \geq 1\right) \in (w_{j+1},\ldots,w_{n}) \times A }\\
 &= \frac{\Pr{\left(W_0,\ldots,W_n \right)=\mathbf{w}} \Pr{\left(W_{n+i} : i \geq 1\right) \in A }}{\Pr{X_1 \geq \gamma n} \Es{\zeta}},
\end{align*} where we use the fact that for every $\mathbf{w}\in G_n$, the term in the above sum is non-zero for the unique value $j=j(\mathbf{w}) \in \llbracket 1,n\rrbracket$ such that $w_{j}-w_{j-1} \geq \gamma n$. We therefore obtain
\[\vert \mu_n(\mathbf{w}\times A)-\nu_n(\mathbf{w}\times A)\vert \leq \Pr{I\geq n} + \left|\frac{\Pr{\zeta \geq n}}{\Es{\zeta} \Pr{X_1\geq \gamma n}}-1\right|, \] which goes to zero as $n\rightarrow \infty$ by \eqref{eq:BB} and the fact that $I$ is almost surely finite. \end{proof}

\begin{proof}[Proof of \cref{lem:PF}]
 First set $\eta_{n}= \sqrt{n \Es{\vert W_n+\gamma n \vert}}$ and recall that $\Yn_{j} = \Zn_j-\Zn_{j-1}$ for every $j \geq 1$. Then, observe that for every fixed $K>0$, as soon as $n \geq K$ the event
\begin{align*}
\{I \leq K\} \cap \left\{ \max_{1 \leq i < I} \Yn_{i}< \gamma n \right\} \cap \left\{ \Yn_{I}> \gamma n+\eta_{n}\right\} 
&\cap \left\{  \min_{0 \leq i \leq n-I} \left(\Zn_{I+i}-\Zn_{I}\right) > -\gamma n-\eta_{n} \right\}\\
& \cap\left\{ \max_{I < i \leq n } \Yn_{i}< \gamma n \right\} 
\end{align*}
is included in the event $\{(\Zn_0, \ldots, \Zn_n )\in G_{n}\}$. Thus, $\nu_{n}(\overline{G_{n}})$ is bounded from above by
\begin{align*}
	   \Pr{I>K} &+ \max_{1\leq j < K}\Pr{\max_{1\leq i \leq j}X_i \geq \gamma n \ \Big| \ \zeta \geq j} +\Pr{X_1 < \gamma n + \eta_n \mid X_1 \geq \gamma n } \\
		& +  \max_{1 \leq j < K} \Pr{ \min_{1\leq i \leq n-j}{W_i}\leq -\gamma n-\eta_{n}}+ \max_{1 \leq j \leq n-K} \Pr{\max_{1 \leq i \leq j} X_{i} \geq \gamma n}.
\end{align*}
Since $\P(I>K) \rightarrow 0$ as $K \rightarrow \infty$, it is enough to show that each one of the  last four terms of the above inequality tends to $0$ as $n \rightarrow \infty$.

\paragraph*{First term.} This follows from the fact that $\max_{1 \leq i \leq K}{X_i} $ is almost surely finite.

\paragraph*{Second term.} Write \begin{equation*}
	 \Pr{ X_1< \gamma n + \eta_{n} \ \middle| \ X_1 \geq \gamma n} =1- \frac{\P(X_1 \geq  \gamma n + \eta_{n})}{\P(X_1 \geq \gamma n)}=  1-\frac{L(\gamma n+\eta_{n})}{L(\gamma n)}   \frac{1}{(1+\eta_{n}/(\gamma n))^{\beta}}
	 \end{equation*}
which tends to $0$ as $n \rightarrow \infty$ since $\theta_n\coloneqq\Es{\vert W_n+\gamma n \vert}/n \rightarrow 0$ by the law of large numbers, and thus $\eta_{n}/n=\sqrt{\theta_{n}} \rightarrow 0$.

\paragraph*{Third term.} Using  Doob's inequality (see e.g.~\cite[Theorem 5.4.2]{Dur10}), write
\[\max_{1 \leq j < K} \Pr{ \min_{1\leq i \leq n-j}{W_i}\leq -\gamma n-\eta_{n}} \leq \Pr{ \min_{1\leq i \leq n}{W_i}\leq -\gamma n-\eta_{n}}\leq \frac{\Es{\vert W_n+\gamma n \vert}}{\eta_n}.\] 
But $\eta_{n}=n\sqrt{\theta_n}$ so we have  $\Pr{\min_{1\leq i \leq n}{W_i}\leq -\gamma n-\eta_{n}}  \leq \sqrt{\theta_{n}}\rightarrow 0$ as $n \rightarrow \infty$. 

\paragraph*{Fourth term.} We have
\[\max_{1 \leq j \leq n-K} \Pr{\max_{1 \leq i \leq j} X_{i} \geq \gamma n } \leq  \Pr{\max_{1 \leq i \leq n} X_{i} \geq \gamma n} \leq n \Pr{X_{1} \geq \gamma n}= \frac{nL(\gamma n)}{(\gamma n)^{\beta}}\]
which tends to $0$ as $n \rightarrow \infty$ since $\beta>1$.
\end{proof}

\subsection{Application to looptrees}

We first state a straightforward consequence of \cref{thm:dTV} (details are left to the reader). If $I$ is an interval, we denote by ${\D}(I, \R)$ the set of real-valued càdlàg functions on $I$ equipped with the Skorokhod $J_{1}$ topology (see \cite[Chapter VI]{JS03} for background).

\begin{proposition}\label{thm:scalinglimit}
Let $(\Wn_{i} : i \geq 0)$ be distributed as the random walk $(W_{i} : i \geq 0)$ under the conditional probability $ \P( \, \cdot \, | \zeta \geq n)$ (and set $\Wn_i=0$ for $i<0$). Let also $J$ be the real-valued random variable such that $\Pr{J \geq x}=(\gamma/x)^{\beta}$ for $x \geq \gamma$. Then, the convergence
\[ \left(  \frac{W^{(n)}_{\lfloor n t \rfloor}}{n} : t \geq -1 \right)  \quad \xrightarrow[n\to\infty]{}  \quad   \left( (J-\gamma t) \mathbbm{1}_{t \geq 0} : t \geq -1\right)\] holds in distribution in $\D([-1,\infty),\R)$.
In addition, the convergence
\[\frac{1}{n}\inf\left\lbrace i\geq 1 : W^{(n)}_i<0 \right\rbrace\quad \xrightarrow[n\to\infty]{}  \quad \frac{J}{\gamma}\] holds jointly in distribution.

\end{proposition}

Observe that instead of working as usual with $\D(\R_{+},\R)$, we work with $\D([-1,\infty),\R)$
by extending our function with value $0$ on $[-1,0)$.
The reason is that our limiting process   almost surely
takes a positive value in $0$ (it ``starts with a jump''),
while $(W^{(n)}_{i} : i \geq 0)$ stays small for a positive time.

Let us mention that this result extends \cite[Theorem 3.2]{Dur80} since it allows infinite variance for the step distribution of the random walk. This possibility is  mentioned in \cite{Dur80}, but the proof of Theorem 3.2 in \cite{Dur80} uses a finite variance condition (see in particular the estimates at the top of page 285 in \cite{Dur80}).

We finally prove Theorem \ref{thm:circle} using Proposition \ref{thm:scalinglimit}.

\begin{proof}[Proof of \cref{thm:circle}] The proof is similar to that of \cite[Theorem 1.2]{CK15}. For every $n\geq 1$, let $\Tgn$ be a $\BGW_\mu$ tree conditioned to have at least $n$ vertices, such that the offspring distribution $\mu$ is subcritical and satisfies $\mu([i,\infty))=L(i)i^{-\beta}$. 

Let us consider the {\L}ukasiewicz path $(\W_i(\Tgn) : i\geq 0)$ of the random tree $\Tgn$. By construction (see for instance \cite[Section 2]{Du03}), $(\W_i(\Tgn) : i\geq 0)$ has the same distribution as a random walk $(W_i : i\geq 0)$ with step distribution $\xi$ defined by \[\xi(i)=\mu(i+1), \quad i\geq -1,\]  conditionally on the event $\{\zeta \geq n\}$. Thus, the requirements of Proposition \ref{thm:scalinglimit} are met by the process $(\W_i(\Tgn) : i\geq 0)$. First, by the correspondence between jumps of the {\L}ukasiewicz path and the degrees of the vertices in  $\Tgn$, there exists an asymptotically unique vertex $v^*_n\in \Tgn$ with maximal degree, such that 
\[\frac{k_{v^*_n}(\Tgn)}{n} \quad \xrightarrow[n\to\infty]{}  \quad J \quad \text{in distribution},\] where $J$ is the random variable defined in Proposition \ref{thm:scalinglimit}. Moreover, if $\{T^*_j : 0\leq j \leq k_{v^*_n}(\Tgn) \}$ are the connected components of $\Tgn\backslash\{v^*_n\}$, we have
\[\frac{1}{n}\sup_{0\leq j \leq k_{v^*_n}(\Tgn)}{\vert T^*_j \vert }\quad \xrightarrow[n\to\infty]{}  \quad 0, \quad \text{in probability},\] see \cite[Corollary 1]{Kor15} for a similar statement. Now, recall the construction of the looptree $\Loop(\tau)$ from a plane tree $\tau$, detailed in Section \ref{sec:Intro}. By construction, $\Loop(\Tgn)$ has a unique face of degree $k_{v^*_n}(\Tgn)+1$, and the largest connected component of $\Loop(\Tgn)$ deprived of this face has size $o(n)$. This completes the proof. \end{proof}

\section{A spinal decomposition}
\label{sec:spinal}

The goal of this section is to establish \cref{thm:trunk2}.

\subsection{A general formula}

Recall from \cref{ss:spinal} the notation $\Trunk(\tau,u)$ when $\tau$ is a tree and $u \in \tau$, the notation $\Lambda(\tau)$ for the number of leaves of $\tau$ and the definition of $\Trunk^{\ast}_{h}$,  the  ``size-biased trunk'' of height $h\geq 0$. Throughout this section, we denote by $ \mathcal{T}_{n}$  a $\BGW_{\mu}$ tree conditioned on having $n$ vertices and by $ \Vcn $ a vertex chosen uniformly at random in $ \mathcal{T}_{n}$.  Let $(W_{n} : n \geq 0)$ be the random walk with jump distribution $\P(W_{1}=i)=\mu(i+1)$ for $i \geq -1$ (started at 0), and set $\phi_{n}(j)\coloneqq\P(W_{n}=-j)$ to simplify notation.

The following result is a simple consequence of the size-biasing relation in \cite{LPP95b}, see also \cite{BM14,Mar16} for similar statements in a different context. 

\begin{proposition}
\label{prop:bias}
For every nonnegative functional $F$ defined on the set of all trees and every integer $h \in \N$,
\[\Es{F(\Trunk( \mathcal{T}_{n}, \Vcn))\mathbbm{1}_{| \Vcn|=h}}=\Es{F(\Trunk^{\ast}_{h}) {\Lambda(\Trunk^{\ast}_{h})}\frac{  \phi_{n-h}(-\Lambda(\Trunk^{\ast}_{h}))}{(n-h)\phi_{n}(-1)}}.\] 
\end{proposition}

\begin{proof}
Let $ \mathcal{T}$ bet a $\BGW_{\mu}$ tree. By a standard size-biasing relation (see e.g.~\cite[Equation~(24)]{Duq09}):
\begin{align*}
\Es{F(\Trunk( \mathcal{T}_{n}, \Vcn))\mathbbm{1}_{| \Vcn|=h}}&= \frac{1}{n}\Es{\sum_{|u|=h}F(\Trunk( \mathcal{T}_{n}, u))}\\
&=\frac{1}{n \P( | \mathcal{T}|=n) } \Es{\sum_{|u|=h}F(\Trunk( \mathcal{T}, u))\mathbbm{1}_{| \mathcal{T}|=n }}\\
&= \frac{1}{n \P( | \mathcal{T}|=n) }\Es{F(\Trunk^{\ast}_{h})  \Psi_{\Lambda(\Trunk^{\ast}_{h})}(n-h)},
\end{align*}
where $\Psi_{k}(n)$ is the probability that $k$ independent $\BGW_{\mu}$ trees have total size $n$. By Kemperman's formula (see e.g.~\cite[Section 6.1]{Pit06}), we have $ \P( | \mathcal{T}|=n)=  \tfrac{1}{n}\P(W_{n}=-1)$ and $\Psi_{k}(n)= \tfrac{k}{n} \P(W_{n}=-k)$ which concludes the proof.
\end{proof}

\subsection{Application to offspring distributions in the domain of attraction of a stable law}

Assume that $\mu$ is a critical offspring distribution belonging to the domain of attraction of a stable law of index $\alpha \in (1,2]$.
This means that if $\sigma_{\mu}^{2} \in (0,\infty]$ is the variance of $\mu$ and if $X_{1}, X_{2}, \ldots$ are i.i.d.\ random variables with distribution $\mu$, then there exists an increasing sequence $(B_{n} : n\geq 1)$  such that  $(X_{1}+\cdots+X_{n}-n)/B_{n}$ converges in distribution to a random variable with Laplace exponent $\lambda \mapsto e^{\lambda^{\alpha}}$ (when $\alpha=2$, this corresponds to $\sqrt{2}$ times a standard Gaussian random variable; in accordance with the convention of \cite{CK14b}). Equivalently, this means that if $X$ is a random variable with distribution $\mu$, there is a slowly varying function $\ell$ such that $\textrm{Var}(X \cdot \mathbbm{1}_{X \leq n})= n^{2-\alpha} \ell(n)$, and
\begin{equation}
\label{eq:Bn} \frac{n\ell(B_{n})}{B_{n}^{\alpha}}  \quad \mathop{\longrightarrow}_{n \rightarrow \infty} \quad  \frac{1}{(2-\alpha)\Gamma(-\alpha)},
\end{equation}
see \cite[Section~2.1]{Kor17} (by continuity, the quantity $((2-\alpha)\Gamma(-\alpha))^{-1}$ is interpreted as equal to $2$ for $\alpha=2$). 

Our goal is now to prove \cref{thm:trunk2}. We start with a preliminary lemma.
Denote by $(X^{\ast}_{i} : i \geq 1)$ a sequence of i.i.d.\ random variables with distribution $\mu^{\ast}$ and set $W^{\ast}_{i}=X_{1}^{\ast}+ \cdots+X_{i}^{\ast}$.
 \begin{lemma}
 \label{lem:drift}
Assume that $\alpha=2$ and denote by $\sigma_{\mu}^{2} \in (0,\infty]$ the variance of $\mu$.
\begin{enumerate}
\item[(i)] The convergence
 \[  \left( \frac{W^{\ast}_{\lfloor t \frac{n}{B_{n}} \rfloor} }{B_{n}} : t \geq 0 \right)  \quad \xrightarrow[n \rightarrow \infty]{} \quad ( 2(1+\sigma_{\mu}^{-2}) t: t \geq 0)\]
 holds in probability, uniformly on compact subsets of $\R_{+}$.
\item[(ii)] The convergence
\[ \left(\frac{\Lambda(\Trunk^{\ast}_{\lfloor t \frac{n}{B_{n}} \rfloor})}{B_{n}} : t \geq 0\right) \quad \xrightarrow[n \rightarrow \infty]{} \quad \left(  2 t:t \geq 0\right)\]
holds in probability, uniformly on compact subsets of $\R_{+}$.
\end{enumerate}
 \end{lemma}
 
 \begin{proof}
It is enough to establish both assertions for $t=1$ (see e.g.~\cite[Theorem 16.14]{Kal02}).

 First assume that $\mu$ has finite variance, so that we may take $B_{n}= \sigma_{\mu}\sqrt{n/2}$. Since $\Es{X_{1}^{\ast}}=\sigma_{\mu}^{2}+1$, we have
 \[\Es{e^{-\lambda X_{1}^{\ast}}}=1-(\sigma_{\mu}^{2}+1)\lambda(1+o(1)), \qquad \lambda \rightarrow 0^+.\]
 Therefore
\[
 \Es{e^{-\frac{\lambda }{B_{n}} W^{\ast}_{\lfloor  \frac{n}{B_{n}} \rfloor} } } = \left( 1-\frac{\lambda }{B_{n}}(\sigma_{\mu}^{2}+1)(1+o(1)) \right)^{\lfloor  \frac{n}{B_{n}} \rfloor}= \exp \left( - \lambda \frac{n}{B_{n}^{2}} (\sigma_{\mu}^{2}+1)(1+o(1))\right).\]
Since $n/B_{n}^{2} \rightarrow {2}/{\sigma_{\mu}^{2}}$,  $(i)$ follows.

Now assume that $\mu$ has infinite variance. By \cite[Equation~(44)]{Kor17}, we have
 \[\Es{e^{-\lambda X_{1}^{\ast}}}=1-\lambda\ell(1/\lambda)(1+o(1)), \qquad \lambda \rightarrow 0.\]
Therefore
\[
 \Es{e^{-\frac{\lambda }{B_{n}} W^{\ast}_{\lfloor  \frac{n}{B_{n}} \rfloor} } } = \left( 1-\frac{\lambda }{B_{n}}\ell\left(\frac{B_{n}}{\lambda}\right)(1+o(1)) \right)^{\lfloor  \frac{n}{B_{n}} \rfloor}= \exp \left( - \lambda \frac{n}{B_{n}^{2}} \ell \left( \frac{B_{n}}{\lambda} \right)(1+o(1))\right),\]
which converges to $ e^{-2\lambda}$ as $ n \rightarrow \infty$,
 by \eqref{eq:Bn} and the fact that $\ell$ is slowly varying. 
 
 The second assertion readily follows from $(i)$ by  observing that we can write $(\Lambda(\Trunk^{\ast}_{h}): h\geq 0)$ as $(W^{\ast}_{h}-h+1 : h\geq 0)$ (and using the fact that $n/B_{n}^{2} \rightarrow 2/\sigma_{\mu}^{2}$ when $\mu$ has finite variance).
 \end{proof}
 
We now turn to the proof of \cref{thm:trunk2}. Recall that $ \Vc_{n} ^{t}$ stands for a vertex chosen uniformly at random among all those at height $\lfloor t \frac{n}{B_{n}} \rfloor$ in $\Tc_n$.
 
\begin{proof}[Proof of \cref{thm:trunk2}]
Fix $\varepsilon>0$.  In order to establish both assertions, we shall estimate $\Es{F(\Trunk( \mathcal{T}_{n},\Vcn))\mathbbm{1}_{|\Vcn|=\lfloor t \frac{n}{B_{n}} \rfloor}}$ uniformly for $ \varepsilon \leq t \leq 1/\varepsilon$ and for all measurable functions $F$ with $\lVert F \rVert_{\infty} \leq 1$. To this end, using  \cref{prop:bias} we write
\begin{align*}
&\Es{F\left(\Trunk( \mathcal{T}_{n}, \Vcn),\lfloor t \frac{n}{B_{n}} \rfloor\right)\mathbbm{1}_{| \Vcn|=\lfloor t \frac{n}{B_{n}} \rfloor}}\\
&  \qquad \qquad \qquad  =\Es{F\left(\Trunk^{\ast}_{\lfloor t \frac{n}{B_{n}} \rfloor},\lfloor t \frac{n}{B_{n}} \rfloor\right) {\Lambda(\Trunk^{\ast}_{\lfloor t \frac{n}{B_{n}} \rfloor})}\frac{  \phi_{n-\lfloor t \frac{n}{B_{n}} \rfloor}(-\Lambda(\Trunk^{\ast}_{\lfloor t \frac{n}{B_{n}} \rfloor}))}{(n-\lfloor t \frac{n}{B_{n}} \rfloor)\phi_{n}(-1)}}.
\end{align*}
Then by combining  \cref{prop:bias} and \cref{lem:drift} with the local limit theorem (see e.g.~\cite[Theorem 4.2.1]{IL71}), in virtue of which we have
\[  \sup_{k \in \Z}  \left| B_{n} \phi_{n}(k)-  \frac{1}{\sqrt{4 \pi}}e^{- \frac{k^{2}}{4 B_{n}^{2}}} \right|  \quad \mathop{\longrightarrow}_{n \rightarrow \infty} \quad 0,\]
 we get
\begin{equation}
\label{eq:unif}
\Es{F\left(\Trunk( \mathcal{T}_{n},\Vcn),\lfloor t \frac{n}{B_{n}} \rfloor\right)\mathbbm{1}_{|\Vcn|=\lfloor t  \frac{n}{B_{n}}\rfloor}} = \frac{B_{n}}{n}  \left( \Es{F\left(\Trunk^{\ast}_{\lfloor t  \frac{n}{B_{n}} \rfloor},\lfloor t \frac{n}{B_{n}} \rfloor\right)} 2 t e^{-t^{2}} +o(1)\right),
\end{equation}
uniformly  for $ \varepsilon \leq t \leq 1/\varepsilon$ and $\lVert F \rVert_{\infty}  \leq 1$ (we also use the fact that $n/B_{n}=o(n)$).

Now observe that by taking $F \equiv 1$ in \eqref{eq:unif}, it follows that
\begin{equation}
\label{eq:llheight}
\P \left( |\Vcn|=\lfloor t  \frac{n}{B_{n}}\rfloor\right)  \quad \mathop{\sim}_{n \rightarrow \infty} \quad  \frac{B_{n}}{n} 2t e^{-t^{2}/2},\end{equation}
uniformly for $ \epsilon \leq t \leq 1/\epsilon$. The first assertion then follows from the simple identity
\[\Es{F( \Trunk( \mathcal{T}_{n},\Vc_{n} ^{t}))}=\Es{ F(\Trunk( \mathcal{T}_{n},\Vcn)) \ \Big| \ {|\Vcn|=\lfloor t  \frac{n}{B_{n}}\rfloor}}.\]

For the second assertion, since $ \mathcal{R}$ is almost surely finite, it is enough to show that
\[\Es{F\left(\Trunk( \mathcal{T}_{n}, \Vcn) \right)\mathbbm{1}_{ \epsilon \frac{n}{B_{n}} \leq | \Vcn| \leq  \frac{1}{\epsilon} \frac{n}{B_{n}} }}-  \int_{\epsilon}^{1/\epsilon} {\d}t \   \Es{F\left(\Trunk^{\ast}_{\lfloor t  \frac{n}{B_{n}} \rfloor}\right)} 2 t e^{-t^{2}}   \quad \mathop{\longrightarrow}_{n \rightarrow \infty} \quad 0,\]
uniformly for $ \varepsilon \leq t \leq 1/\varepsilon$ and measurable functions $F$ with $\lVert F \rVert_{\infty} \leq 1$. This is also a consequence of \eqref{eq:unif}, as we have
\begin{align*}
\Es{F\left(\Trunk( \mathcal{T}_{n}, \Vcn) \right)\mathbbm{1}_{ \epsilon \frac{n}{B_{n}} \leq | \Vcn| \leq  \frac{1}{\epsilon} \frac{n}{B_{n}} }}
& = \sum_{k=\lfloor \epsilon \frac{n}{B_{n}} \rfloor}^{ \lfloor \frac{1}{\epsilon} \frac{n}{B_{n}} \rfloor}\Es{F\left(\Trunk( \mathcal{T}_{n}, \Vcn) \right)\mathbbm{1}_{ | \Vcn| =k}} \\
&= \int_{\epsilon+o(1)}^{1/\epsilon+o(1)} {\d}t \ \frac{n}{B_{n}}  \Es{F(\Trunk( \mathcal{T}_{n},\Vcn))\mathbbm{1}_{|\Vcn|=\lfloor t  \frac{n}{B_{n}}\rfloor}}  \\
&= o(1)+ \int_{\epsilon}^{1/\epsilon} {\d}t \   \Es{F\left(\Trunk^{\ast}_{\lfloor t  \frac{n}{B_{n}} \rfloor},\lfloor t \frac{n}{B_{n}} \rfloor\right)} 2 t e^{-t^{2}} .
\end{align*}
This completes the proof.
\end{proof}

\begin{remark}[Case $1<\alpha<2$] When $1<\alpha<2$, the conclusions of \cref{thm:trunk2} are not true. Indeed, in this case,  $ \frac{1}{B_{n}}\Lambda(\Trunk^{\ast}_{\lfloor t \frac{n}{B_{n}} \rfloor})$ converges in distribution to a non-degenerate random variable (a spectrally positive stable random variable with index $\alpha-1$), which implies that asymptotically the degrees on the spine of $\Trunk( \mathcal{T}_{n},\Vc_{n} ^{t})$ are not i.i.d.\ anymore, but biased by a functional involving the total number of leaves. However, still using the local limit theorem,  it remains possible to show that if $(F_{n} : n \geq 1)$ is a sequence of measurable functions such that $\lVert F_{n} \rVert \leq 1$ for every $n \geq 1$,
\begin{enumerate}
\item[(i)] for fixed $t>0$, if $\Es{F_{n}(\Trunk^{\ast}_{\lfloor t  \frac{n}{B_{n}} \rfloor })} \rightarrow 0$, then $\Es{F_{n}( \Trunk( \mathcal{T}_{n},\Vc_{n} ^{t}))} \rightarrow 0$;
\item[(ii)] if $\Es{F_{n}(\Trunk^{\ast}_{\lfloor  \mathcal{R}  \frac{n}{B_{n}} \rfloor })} \rightarrow 0$, then $\Es{F_{n}(\Trunk( \mathcal{T}_{n}, \Vc ))} \rightarrow 0$.
\end{enumerate}
It also possible to obtain a local limit theorem for the height $| \Vc|$ of a uniformly chosen vertex in $ \mathcal{T}_{n}$  in the spirit of \eqref{eq:llheight}, namely:
 \[\P \left( |\Vc|=\lfloor t  \frac{n}{B_{n}}\rfloor\right)  \quad \mathop{\sim}_{n \rightarrow \infty} \quad  \frac{B_{n}}{n} \Es{X_{t}^{(\alpha-1)}p^{(\alpha)} \left( -X_{t}^{(\alpha-1)} \right)}\]
 uniformly for $t$ in compact subsets of $\R_{+}^{\ast}$, where $X_{t}^{(\alpha-1)}$ is an $(\alpha-1)$-stable random variable with Laplace exponent $\Es{e^{-\lambda X_{t}^{(\alpha-1)}}}=e^{\alpha t \lambda^{\alpha-1}}$ and $p^{(\alpha)}$ is the density of a random variable $Y$ with Laplace transform $\Es{e^{-\lambda Y}}=e^{\lambda^{\alpha}}$ (observe that for $\alpha=2$, we indeed have $X_{t}^{(\alpha-1)}p^{(\alpha)} \left( -X_{t}^{(\alpha-1)} \right)=2te^{-t^{2}}$ almost surely).
We leave the details to the reader.
\end{remark}

\section{Looptrees: the Gaussian domain of attraction}
\label{sec:Th2}

We now study the asymptotic behavior of looptrees associated with large  BGW trees whose offspring distribution has mean 1 and belongs to the domain of attraction of a Gaussian distribution. In particular, we establish \cref{thm:CRT}.

If $\tau$ is a plane tree, recall from Section \ref{sec:Looptreesdef} the definition of $\Loop( \tau)$ (and see Figure \ref{fig:loopintro} for an example). Recall also from \cref{sssec:coding} that $(\W_{t}(\tau) : t \geq 0)$, $(\C_{t}(\tau) : t \geq 0)$ and $(\H_{t}(\tau) : t \geq 0)$ denote respectively the {\L}ukasiewicz path, contour and height function of $\tau$. Starting from now, we denote by $\dc_{\tau}$ the graph metric on ${\Loop}( \tau)$. If  $\varnothing=u_{0}, u_{1}, \ldots, u_{|\tau|-1}$ denote the vertices of $\tau$ listed in lexicographical order, we set \begin{equation}\label{eqn:DefHrond}
	\Hc_{i}(\tau)=\dc_{\tau}(\varnothing,u_{i}), \quad  0 \leq i < |\tau|,
\end{equation} and $\Hc_{i}(\tau)=0$ for $i\geq |\tau|$. We again extend $\Hc(\tau)$ to $\R_{+}$ by linear interpolation. 

Throughout this section, we fix a critical offspring distribution $\mu$ belonging to the domain of attraction of a Gaussian distribution, and denote by $\Tn$ a $\BGW_{\mu}$ tree conditioned on having $n$ vertices. The main step in the proof of \cref{thm:CRT} is a functional invariance principle for $\Hc(\Tn)$. At some point we shall treat the finite variance and infinite variance cases separately, since in the first case $\Hc(\Tn)$ will be of the same order as $\H(\Tn)$, while in the second case $\Hc(\Tn)$ will be of the same order as $\W(\Tn)$.

\subsection{Towards the proof of \cref{thm:CRT}.} The goal of this section is to establish \cref{thm:CRT}. We now assume that $\mu$ is a critical offspring distribution with finite positive variance $\sigma_\mu^2$. Finally, we recall from \eqref{eqn:CstSigma} the definition of $c_\mu$.

\begin{proposition}
\label{prop:cvjointe}
The convergence
\begin{equation}
\label{eq:cvjointe} \bigg( \frac{B_{n}}{n} \C_{2nt}(\Tn), \frac{B_{n}}{n} \H_{nt}(\Tn),\frac{1}{B_{n}} \Hc_{nt}(\Tn) \bigg)_{0 \leq  t \leq 1}  \quad \xrightarrow[n\to\infty]{}  \quad  \sqrt{2} \cdot \bigg( \mathbbm{e}_{t},\mathbbm{e}_{t},  c_{\mu} \cdot \mathbbm{e}_{t}\bigg)_{0 \leq t \leq 1}
\end{equation}
holds jointly in distribution in the space $\mathcal{C}([0,1])^{3}$, where $\mathcal{C}([0,1])$ is the space of continuous real-valued functions on $[0,1]$ equipped with the uniform norm. 
\end{proposition}

To this end, our main input will be the following result.
\begin{proposition}[\cite{Du03}]
\label{prop:cve}
The convergence
\[\bigg( \frac{1}{B_{n}} \W_{nt}(\Tn), \frac{B_{n}}{{n}} \H_{nt}(\Tn) , \frac{B_{n}}{{n}} \C_{2nt}(\Tn) \bigg)_{0 \leq  t \leq 1}  \quad \xrightarrow[n\to\infty]{}  \quad  \sqrt{2}\cdot(\mathbbm{e}_t,\mathbbm{e}_t,\mathbbm{e}_t)_{0 \leq t \leq 1}
\]
holds in distribution in the space $\mathcal{C}([0,1])^{3}$, where $(\mathbbm{e}_t : 0\leq t \leq 1)$ is the normalized Brownian excursion.
\end{proposition} The normalized Brownian excursion may be seen as Brownian motion conditioned to return to $0$ at time $1$ and to stay positive on $(0,1)$, see \cite[Section 2]{LG05}. This result was established in \cite{MM03} when $\mu$ has small exponential moments, and in \cite{Du03} in the general case (see also \cite{Kor13}). In view of future use, we record the following simple consequence of Proposition~\ref{prop:cve}:
\begin{equation}\label{eqn:DegreMax}
	\frac{1}{B_{n}}\sup_{v \in \Tc_n}{k_v(\Tc_n)} \quad \xrightarrow[n \rightarrow \infty]{} \quad 0, \quad \text{in probability}.
\end{equation}
Indeed, the maximum out-degree of $\Tn$ is the largest jump of $\W(\Tn)$ (plus one). In addition,  \cref{prop:cve}  entails that the convergence $\big(\tfrac{1}{B_{n}} \W_{\lfloor nt \rfloor }(\Tn) : 0 \leq t \leq 1 \big) \rightarrow  \sqrt{2} \cdot (\mathbbm{e}_t : 0\leq t \leq 1)$ holds  in distribution in the space  $\mathcal{D}([0,1])$ of real-valued càdlàg functions on $[0,1]$ equipped with the Skorokhod $J_{1}$ topology. The claim follows from the continuity of the functional ``largest jump'' for the Skorokhod $J_{1}$ topology (see e.g.\ \cite[Proposition 2.4 in Chapter VI]{JS03}).

To prove  \cref{prop:cvjointe}, we establish the following result which, roughly speaking, shows tightness and identifies the finite dimensional marginals. 

\begin{lemma}
\label{lem:tightfd}
For every $n \geq 1$, let $\Tn$ be a $\BGW_{\mu}$ tree conditioned on having $n$ vertices. Then, the following assertions hold.
\begin{enumerate}
\item The sequence $\big(\tfrac{1}{B_{n}}\Hc_{nt} ( \Tn): 0 \leq t \leq 1\big)$ is tight in $ \mathcal{C}([0,1])$.
\item For every $n \geq 1$, let $U^{n}$ be a random variable uniformly distributed on $ \{0,1, \ldots,n-1\}$, independent of $ \mathcal{T}_{n}$. Then,
\[\frac{1}{B_{n}} \big| \Hc_{U^{n}}( \Tn)- c_{\mu} \W_{U^{n}}( \Tn) \big|  \quad \xrightarrow[n\to\infty]{}  \quad  0 \qquad \text{in probability}.\]
\end{enumerate}
\end{lemma}

Before proving this, let us explain why it implies \cref{prop:cvjointe}.

\begin{proof}[Proof of \cref{prop:cvjointe} using \cref{lem:tightfd}]
By \cref{prop:cve} and \cref{lem:tightfd}, up to extraction and using Skorokhod's representation theorem (see e.g.\ \cite[Theorem 6.7]{Bil99}), we may assume that the convergence 
\begin{equation}
\label{eq:Z}\bigg( \frac{1}{B_{n}} \W_{nt}(\Tn), \frac{B_{n}}{n} \C_{2nt}(\Tn), \frac{B_{n}}{n} \H_{nt}(\Tn),\frac{1}{B_{n}} \Hc_{nt}(\Tn) \bigg)_{0 \leq  t \leq 1}  \quad \xrightarrow[n\to\infty]{}  \quad  \sqrt{2} \cdot \bigg( \mathbbm{e}_{t},\mathbbm{e}_{t}, \mathbbm{e}_{t},Z_{t}\bigg)_{0 \leq t \leq 1}
\end{equation}
holds  almost surely in the space $\mathcal{C}([0,1])^{3}$ for a certain continuous random function $Z$, and we aim at showing that \[(Z_t : 0\leq t\leq 1) =  \left(  c_{\mu}\cdot \mathbbm{e}_t : 0\leq t\leq 1\right) \quad \text{almost surely}.\] For every $n \geq 1$, we  let $(U_{i}^{n} : i \geq 1)$ be an i.i.d.~sequence of uniform random variables on $ \{0,1, \ldots,n-1\}$, $(U_{i} : i \geq 1)$ an i.i.d.~sequence of uniform random variables on $[0,1]$, all independent of $(\Tc_n : n \geq 1)$. We may also assume that for every $i \geq 1$, the convergence $\tfrac{1}{n}U_{i}^{n} \rightarrow U_{i}$  holds almost surely as $n \rightarrow \infty$.

Now, let us fix $k \geq 1$. We claim that \[Z_{U_{k}}=  c_{\mu} \cdot \mathbbm{e}_{U_{k}}\quad \text{almost surely}.\] Indeed, by \cref{lem:tightfd} \textit{(ii)},
we may find an extraction $\phi$ such that 
\[ \frac{1}{B_{\phi(n)}} \Big| \Hc_{U_{k}^{\phi(n)}}( \Tc_{\phi(n)})- c_\mu \W_{U_{k}^{\phi(n)}}( \Tc_{\phi(n)}) \Big|  \quad \xrightarrow[n\to\infty]{}  \quad 0 \quad \text{almost surely}.\] But we also have the almost sure convergences
\[ \frac{1}{B_{\phi(n)}} \Hc_{U_{k}^{\phi(n)}}( \Tc_{\phi(n)})  \quad \xrightarrow[n\to\infty]{}  \quad   \sqrt{2} \cdot Z_{U_{k}} \quad \text{and} \quad   \frac{c_\mu}{\sqrt{\phi(n)}} \H_{U_{k}^{\phi(n)}}( \Tc_{\phi(n)})  \quad \xrightarrow[n\to\infty]{}  \quad \sqrt{2} c_{\mu} \cdot\mathbbm{e}_{U_{k}},\]
which entails our claim.

It follows that almost surely, the two continuous functions $(Z_t : 0\leq t \leq 1)$ and $ \left(c_{\mu} \cdot \mathbbm{e}_t : 0\leq t \leq 1\right)$  coincide on the set $ \{U_{i}: i \geq 1\}$ which is dense in $[0,1]$. This completes the proof.
\end{proof}

\subsection{Tightness.} The goal of this section is to establish the tightness statement \textit{(i)} of Lemma~\ref{lem:tightfd}. We  start  with a (deterministic) upper bound for the distance $\dc_{\tau}$.

\begin{lemma}
\label{lem:particulierementutile}
Let $\tau$ be a plane tree and denote by  $u_{0}, u_{1}, \ldots, u_{|\tau|-1}$ its vertices listed in lexicographical order. Then, for every $0 \leq i < j < |\tau|$, if $u_{i}$ is an ancestor of $u_{j}$ in $\tau$ we have
\[| \Hc_i(\tau)-\Hc_j(\tau)| \leq \W_{j}(\tau)-\W_{i}(\tau)+\H_{j}(\tau)-\H_{i}(\tau).\]
\end{lemma}

\begin{proof}
To simplify, assume that $i \neq 0$ (the case $i=0$ is treated in the same way). Observe that $u_{i}$ disconnects $\Loop(\tau)$ into two connected components, one containing $u_{j}$ and the other containing $u_{0}$. We consider the first of these two components, which is actually $\Loop(\theta_{u_i}(\tau))$ (where we recall that $\theta_{u_i}(\tau)$ is the tree made of the descendants of $u_i$ in $\tau$). Since the shortest path from $u_{j}$ to $u_{i}$ in $\Loop(\tau)$ stays in $\Loop(\theta_{u_i}(\tau))$, the distance between $u_i$ and $u_j$ in $\Loop(\theta_{u_i}(\tau))$ is $\Hc_j(\tau)-\Hc_i(\tau)$. On the other hand, the number of vertices branching (weakly) to the right of the ancestral line $\llbracket u_{i},u_j \llbracket$ in $\theta_{u_i}(\tau)$ is $\W_{j}(\tau)-\W_{i}(\tau)+\H_{j}(\tau)-\H_{i}(\tau)$. Therefore  the path in $\Loop(\tau)$ which goes from $u_{j}$ to $u_{i}$ by only using the vertices of $\tau$ belonging to $\llbracket u_{i},u_{j} \llbracket$, and their children grafted on the right of  $\llbracket u_{i},u_{j} \llbracket$ 
 has length $\W_{j}(\tau)-\W_{i}(\tau)+\H_{j}(\tau)-\H_{i}(\tau)$ (see Figure \ref{fig:loopanc} for an illustration). This entails the desired result.
\end{proof}

 \begin{figure}[!h]
 \begin{center}
 \includegraphics[width= .4 \linewidth]{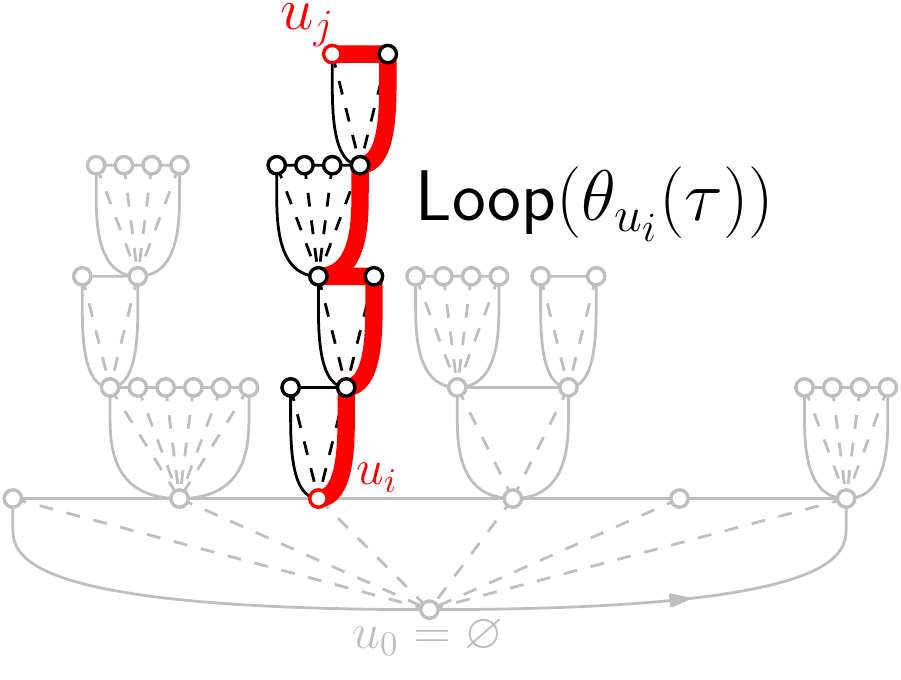}
 \caption{ \label{fig:loopanc}A plane tree $\tau$, the associated looptree $ \mathsf{Loop}( \theta_{u_{i}}(\tau))$ (in black) and the path of length $\W_{j}(\tau)-\W_{i}(\tau)+\H_{j}(\tau)-\H_{i}(\tau)$ between $u_j$ and $u_i$ (in bold).}
 \end{center}
 \end{figure}

We can now prove the tightness statement \textit{(i)} of Lemma \ref{lem:tightfd}.

\begin{proof}[Proof of \cref{lem:tightfd} (i)]
By a standard tightness criterion (see e.g.\ \cite[Theorem 8.10.5]{Dur10}) it suffices to check that for every $\varepsilon>0$,
\[  \limsup_{n \rightarrow \infty} \P\bigg( \sup_{ |i-j| \leq \delta n} |\Hc_{i}(\Tn)-\Hc_{j}(\Tn)|>\varepsilon B_{n} \bigg)  \quad \xrightarrow[\delta \rightarrow 0]{}  \quad 0.\]
Since $\tfrac{B_{n}}{n}\sup \H(\Tn)$ converges in distribution as $n \rightarrow \infty$ by Proposition \ref{prop:cve}, it is enough to check that
\[  \limsup_{n \rightarrow \infty} \P\bigg( \sup_{ |i-j| \leq \delta n} |\Hc_{i}(\Tn)-\Hc_{j}(\Tn)|>\varepsilon B_{n}, \quad  \sup \H(\Tn) \leq n^{3/4}  \bigg)  \quad \xrightarrow[\delta \rightarrow 0]{}  \quad 0.\]

To this end, we start with an identity inspired from  \cite[Equation (11)]{BM14}, see  the proof of Proposition 7 in \cite{Mar16} for a similar argument in a different context.  Let us introduce some notation. We fix $n \geq 1$, and let $\Ttc_n$ be the mirror image of the tree $\Tc_n$ (see Figure \ref{fig:Mirror}). We claim that  on the event $\{\sup \H(\Tn) \leq n^{3/4}\}$, for every $\delta>0$ and every $n$ sufficiently large,
\begin{align}
	&\sup_{ |p-q| \leq \delta n} |\Hc_{p}(\Tn)-\Hc_{q}(\Tn)| \notag\\ 
	& \qquad  \qquad \leq \sup_{ \substack{|i-j| \leq \delta n\\ u_{i} \prec u_{j}}} |\Hc_{i}(\Tn)-\Hc_{j}(\Tn)|+\sup_{ \substack{|i-j| \leq 2\delta n\\ \tilde{u}_{i} \prec \tilde{u}_{j}}} |\Hc_{i}(\Ttc_n)-\Hc_{j}(\Ttc_n)|+ \sup_{v \in \Tc_n}{k_v(\Tc_n)}, \label{eq:bornejuste}
\end{align}
where $\tilde{u}_0=\varnothing, \ldots, \tilde{u}_{n-1}$ denote the vertices of $\Ttc_n$ listed in  the lexicographical order.

To establish \eqref{eq:bornejuste}, we fix $p,q \in \{0, \ldots, n-1\}$ and assume, without loss of generality, that $p<q$. We denote by $\mathsf{m}(p,q)$ the index of the most recent common ancestor between $u_p$ and $u_q$ (in the lexicographical order of $\Tc_n$). We also let $p'$ and $q'$ be the indices of the children of $u_{\mathsf{m}(p,q)}$ that are ancestors of respectively $u_p$ and $u_q$. By the triangular inequality,
\[|\Hc_{p}(\Tn)-\Hc_{q}(\Tn)| \leq |\Hc_{p}(\Tn)-\Hc_{p'}(\Tn)|+|\Hc_{q}(\Tn)-\Hc_{q'}(\Tn)|+  \sup_{v \in \Tc_n}{k_v(\Tc_n)}.\]
One has now to be careful because $u_q$ and $u_{q'}$ are close in the lexicographical order of $\Tc_n$ (since $|q-q'|\leq |p-q|$, see  Figure \ref{fig:Mirror}), but $u_p$ and $u_{p'}$ may not be. However, the indices of their mirror images $\mathsf{I}(u_p)$ and $\mathsf{I}(u_{p'})$ are at distance at most $|p-q|+ |\H_{p}(\Tn)-\H_{p'}(\Tn)|$ in the lexicographical order of $\Ttc_n$ by construction (see Figure \ref{fig:Mirror} for an illustration). Hence, on the event $\{\sup \H(\Tn) \leq n^{3/4}\}$, for every $\delta>0$ and every $n$ sufficiently large,  if $|p-q| \leq \delta n$, then  the mirror images $\mathsf{I}(u_p)$ and $\mathsf{I}(u_{p'})$ are at distance less than $2 \delta n$  in the lexicographical order of $\Ttc_n$.  This entails \eqref{eq:bornejuste}. 

Since $\Ttc_{n}$ and $\Tn$ have the same distribution, by \eqref{eqn:DegreMax}, it  suffices to check that   
\begin{equation}
\label{eq:module}
\limsup_{n \rightarrow \infty} \P \left(  \sup_{ \substack{|i-j| \leq \delta n\\ u_{i} \prec u_{j}}} |\Hc_{i}(\Tn)-\Hc_{j}(\Tn)|>\varepsilon B_{n}\right)  \quad \xrightarrow[\delta \rightarrow 0]{}  \quad 0.
\end{equation}
But by  \cref{lem:particulierementutile}, if $u_{i}$ is an ancestor of $u_{j}$, we have
\[ |\Hc_{i}(\Tn)-\Hc_{j}(\Tn)|=\Hc_{j}(\Tn)-\Hc_{i}(\Tn) \leq \W_{j}(\Tn)-\W_{i}(\Tn)+ \H_{j}(\Tn)-\H_{i}(\Tn),\]
so that the probability in \eqref{eq:module} is bounded from above by 
\[ \P\bigg( \sup_{ |i-j| \leq \delta n} |\W_{i}(\Tn)-\W_{j}(\Tn)|> \frac{\varepsilon}{2} B_{n} \bigg)+\P\bigg( \sup_{ |i-j| \leq \delta n} |\H_{i}(\Tn)-\H_{j}(\Tn)|> \frac{\varepsilon}{2}  B_{n}\bigg).\]
By \cref{prop:cve}, $(\tfrac{1}{B_{n}} \W_{nt}(\Tn): 0 \leq t \leq 1)$ and $(\tfrac{1}{B_{n}} \H_{nt}(\Tn): 0 \leq t \leq 1)$ are both tight in  $\mathcal{C}([0,1])$ (when $\mu$ has infinite variance, we have $\frac{n}{B_{n}}= o( B_{n})$). The desired result then follows. \end{proof}

 \begin{figure}[!h]
 \begin{center}
 \includegraphics[scale=  1]{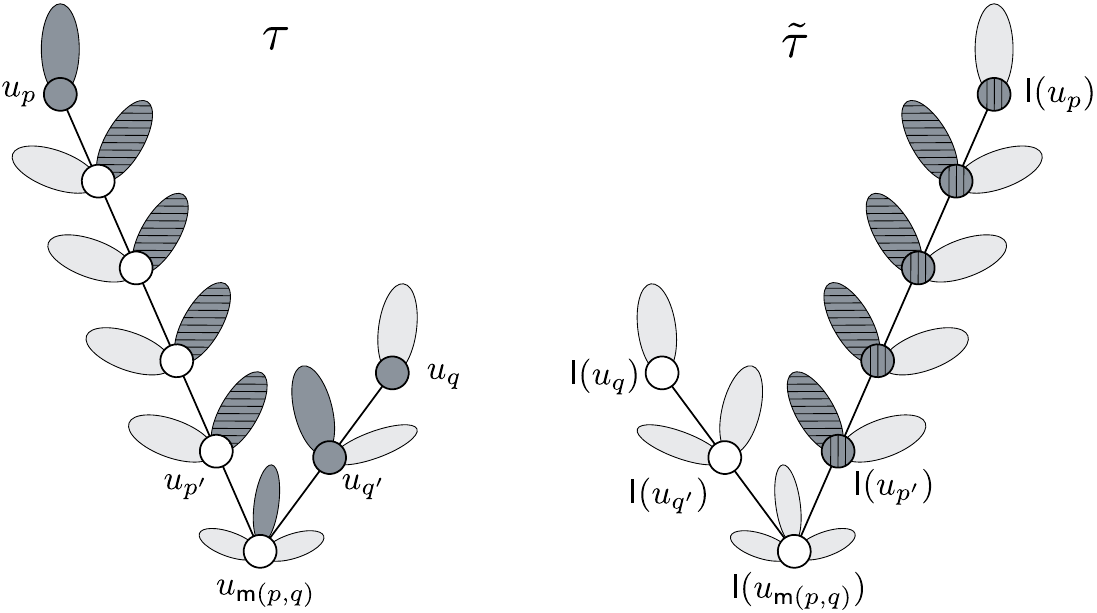}
 \caption{ \label{fig:Mirror}Left: a plane tree $\tau$, where the dark gray region encompasses the vertices that contribute to the lexicographical distance between $u_{p}$ and $u_{q}$ in $\tau$. Right: the mirror $ \tilde{\tau}$ of $\tau$, where the dark gray region encompasses this time the vertices that contribute to the lexicographical distance between $\mathsf{I}(u_{p'})$ and $\mathsf{I}(u_{p})$ in $\tilde{\tau}$.  The indices of $\mathsf{I}(u_{p'})$ and $\mathsf{I}(u_{p})$ in the lexicographical order of $\tilde{\tau}$ are at distance at most $|p-q|+ |\H_{p}(\Tn)-\H_{p'}(\Tn)|$. Indeed, the dark gray region on the right can be decomposed into two parts: the first one (horizontal hatches), which is a subset of the dark gray region on the left, yields to the contribution $|p-q|$, and the second one (vertical hatches) gives the contribution $|\H_{p}(\Tn)-\H_{p'}(\Tn)|$.}
 \end{center}
 \end{figure}

\subsection{Convergence of finite-dimensional distributions}
\label{ss:fd}

The goal of this section is to establish the convergence of finite dimensional marginals statement \textit{(ii)} of Lemma \ref{lem:tightfd}. In what follows, given a tree $\tau$ and a random variable $U$ uniformly distributed on $ \{0,1, \ldots,|\tau|-1\}$, we may interpret the $U$-th vertex of $\tau$ in lexicographical order as a vertex $\Vc$ chosen uniformly at random in $\tau$. Recalling the definition \eqref{eqn:DefHrond}, we then have 
\[\Hc_U(\tau)=\dc_\tau(\varnothing,\Vc).\] We shall use the following result.

\begin{lemma}
\label{lem:SpinalDecomposition} For every $n \geq 1$, let $\Tn$ be a $\BGW_{\mu}$ tree conditioned on having $n$ vertices. Conditionally given $\Tc_{n}$, let $\Vcn$ be a vertex of $\Tc_{n}$ chosen  uniformly at random.  Then
\[\frac{\dc_{\Tc_{n}}(\varnothing, \Vcn )}{R(\Vcn)}  \quad \xrightarrow[n \rightarrow \infty]{} \quad  c_\mu\] in probability, where  $R(\Vcn)$ is the number of children of vertices of $\llbracket \varnothing, \Vcn \llbracket$ branching on the right of $\llbracket \varnothing, \Vcn \llbracket$.
\end{lemma}

Before proving this result, let us explain why it implies \cref{lem:tightfd} \textit{(ii)}. 

\begin{proof}[Proof of \cref{lem:tightfd} (ii) using \cref{lem:SpinalDecomposition}] We interpret the $U^{n}$-th vertex of $\Tc_n$ in lexicographical order as a vertex $\Vcn$ chosen uniformly at random in $\Tc_n$.   If $u_{0}, \ldots,u_{n-1}$ are the vertices of $\Tc_{n}$ listed in lexicographical order, it is well known that for every $0 \leq i < n$, $\W_{i}(\Tc_{n})$ is equal to the number of children of vertices of $\llbracket \varnothing, u_{i} \llbracket$ branching on the right of $\llbracket \varnothing, u_{i} \llbracket$, so that $\dc_{\Tc_{n}}(\varnothing, \Vcn) =\Hc_{U^{n}}(\Tc_{n})$ and $R(\Vcn)=\W_{U^{n}}(\Tc_{n})$.

Then fix $\epsilon>0$ and write
\[ \left\{ \big| \Hc_{U^{n}}( \Tn)- c_{\mu} \W_{U^{n}}( \Tn) \big|  \geq \varepsilon B_{n} \right\} = \left\{ \left|\frac{\dc_{\Tc_{n}}(\varnothing, \Vcn )}{R(\Vcn)} - c_{\mu} \right|  \geq \epsilon \frac{B_{n}}{\W_{U^{n}}(\Tc_{n})}\right\}.\]
By \cref{prop:cve},  $ \frac{\W_{U^{n}}(\Tc_{n})}{B_{n}}$ converges in distribution to $\sqrt{2}\mathbbm{e}_{U}$ with $U$ uniform on $[0,1]$. By \cref{lem:SpinalDecomposition} it readily follows that $\P( \big| \Hc_{U^{n}}( \Tn)- c_{\mu} \W_{U^{n}}( \Tn) \big|  \geq \varepsilon B_{n}) \rightarrow 0$. This completes the proof.
\end{proof}

The key ingredient to prove \cref{lem:SpinalDecomposition} is the  spinal decomposition of \cref{sec:spinal}.

\begin{proof}[Proof of  \cref{lem:SpinalDecomposition}]
We separate the proof depending on the finiteness of $\sigma_{\mu}^{2}$, since when $\sigma_{\mu}^{2}<\infty$ the quantities $\dc_{\Tc_{n}}(\varnothing, \Vcn )$ and $R(\Vcn)$ are of the same order as $|\Vcn|$, unlike in the case $\sigma_{\mu}^{2}=\infty$.

Recall the notation of \cref{sec:spinal}. For every $n \geq 1$, denote by $\varnothing=v^{\ast}_{0},v^{\ast}_{1}, \ldots,v^{\ast}_{h}$ the vertices on the spine of $\Trunk^{\ast}_{h}$. Observe also that the quantities  $\dc_{\Tc_{n}}(\varnothing, \Vcn )$ and $R(\Vcn)$ only depend on $ \Trunk( \mathcal{T}_{n}, \Vcn )$. Therefore, by \cref{thm:trunk2}, it is enough to show that
\begin{equation}
\label{eq:lllgoal}\frac{\dc_{\Trunk^{\ast}_{n}}(\varnothing, v^{\ast}_{n} )}{R(v^{\ast}_{n})}  \quad \xrightarrow[n \rightarrow \infty]{} \quad  c_\mu
\end{equation} in probability, where  $R(v^{\ast}_{n})$ is the number of children of vertices of $\llbracket \varnothing, v^{\ast}_{n} \llbracket$ branching on the right of $\llbracket \varnothing, v^{\ast}_{n} \llbracket$.

\paragraph*{Case $\sigma_{\mu}^{2}<\infty$.} Observe that by construction, $\dc_{\Trunk^{\ast}_{n}}(\varnothing, v^{\ast}_{n} )$ is a sum of $n$ i.i.d.\ random variables whose distribution is that of $\min(U_{X^{\ast}},X^{\ast}-U_{X^{\ast}}+1)$, where $X^{\ast}$ has the size-biased law $\mu^{\ast}$ and conditionally on $X^{\ast}$, $U_{X^{\ast}}$ is uniform on $\{1,\ldots,X^{\ast}\}$. One sees that such a variable has mean $c_\mu$, so that by the law of large numbers, in probability,
\begin{equation}
\label{eq:lll1}
\frac{\dc_{\Trunk^{\ast}_{n}}(\varnothing, v^{\ast}_{n} )}{n} \  \quad \xrightarrow[n \rightarrow \infty]{} \quad  c_\mu.
\end{equation}
On the other hand, $R(v^{\ast}_{n})$ is the sum of $n$ i.i.d.~random variable with distribution $\mu$, 
so that by the law of large numbers once again,
\begin{equation}
\label{eq:lll2}
\frac{R(v^{\ast}_{n})}{n} \  \quad \xrightarrow[n \rightarrow \infty]{} \quad  1
\end{equation} in probability. The convergence \eqref{eq:lllgoal} then follows from \eqref{eq:lll1} and \eqref{eq:lll2}.

\paragraph*{Case $\sigma_{\mu}^{2}=\infty$.} For $1 \leq i \leq n$, denote by  $L_{i}$ (resp.~$R_{i}$) as the number of children of $v^{\ast}_{i}$ branching on the left (resp.~right) of $\llbracket \varnothing, v^{\ast}_{i}\llbracket$ (so that $\min(L_{i}+1,R_{i}+1)$ is $\dc_{\Trunk^{\ast}_{n}}(v^{\ast}_{i-1},v^*_{i})$). We shall establish that
\begin{equation}
\label{eq:cvinfini}\frac{\displaystyle \sum_{i=1}^{n} \min(L_{i}+1,R_{i}+1)}{ \displaystyle \sum_{i=1}^{n} R_{i}}  \quad \xrightarrow[n \rightarrow \infty]{} \quad  \frac{1}{2}
\end{equation}
which  is equivalent to \eqref{eq:lllgoal}.  To this end, observe that by construction $((L_{i},R_{i}) : i \geq 1)$ is a sequence of independent random variables with distribution given by $\P((L_{1},R_{1})=(i,j))=\mu(i+j+1)$ for $i,j \geq 0$.

We now recall some  results concerning random variables falling within the domain of attraction of a stable law of index $\alpha=1$ (see e.g.~\cite{Ber17}). Assume that $(Z_{i} : i \geq 1)$ are i.i.d.~integer valued random variables such that
\[\P(Z_{1} \geq k)= \frac{\ell(k)}{k}, \qquad k \geq 1,\]
where $\ell$ is a slowly varying function such that $\sum_{k \geq 1} \tfrac{\ell(k)}{k}=\infty$ (so that $\E(Z_{1})=\infty)$. Then the convergence
\[ \frac{Z_{1}+ \cdots+Z_{n}}{n  \cdot \sum_{k=1}^{a_{n}} \frac{\ell(k)}{k}}  \quad \mathop{\longrightarrow}_{n \rightarrow \infty} \quad 1\]
holds in probability, where $a_{n}$ satisfies $ \tfrac{\ell(a_{n})}{a_{n}} \sim \tfrac{1}{n}$ as $n \rightarrow \infty$.

Back to the proof of \eqref{eq:cvinfini}, observe that for every $k \geq 0$, $\P(R_{1}=k)=\mu([k+1,\infty))$ and $\P(\min(L_{1},R_{1})=k)=2\mu([2k+1,\infty))-\mu(2k+1)$. Therefore, by standard integration properties of slowly varying functions (see e.g.~\cite[Proposition 1.5.10]{BGT89})
\[\P(R_{1} \geq k)  \quad \mathop{\sim}_{k \rightarrow \infty} \quad  \frac{L(k)}{k}, \qquad \P(\min(L_{1},R_{1}) \geq k)  \quad \mathop{\sim}_{k \rightarrow \infty} \quad  \frac{L(k)}{2k}.\]
As a consequence if we choose $a_{n}$ so that $\tfrac{L(a_{n})}{a_{n}} \sim \frac{1}{n}$, by the previous paragraph the convergences
\[
\frac{ \displaystyle\sum_{i=1}^{n} R_{i}}{n  \cdot \displaystyle\sum_{k=1}^{a_{n}} \frac{L(k)}{k}}  \quad \mathop{\longrightarrow}_{n \rightarrow \infty} \quad 1 \qquad \text{and} \qquad \frac{ \displaystyle\sum_{i=1}^{n} \min(L_{i},R_{i})}{\displaystyle n  \cdot \sum_{k=1}^{a_{n}/2} \frac{L(k)}{2k}}  \quad \mathop{\longrightarrow}_{n \rightarrow \infty} \quad 1
\]
hold in probability.
Since $\sum_{k=1}^{a_{n}} \frac{L(k)}{k} \rightarrow \infty$ as $n \rightarrow \infty$, it follows that \[ \frac{\displaystyle \sum_{i=1}^{n} \min(L_{i},R_{i})}{\displaystyle \sum_{i=1}^{n} \min(L_{i}+1,R_{i}+1)}  \quad \mathop{\longrightarrow}_{n \rightarrow \infty} \quad 1 \qquad \text{in probability}.\] It therefore remains to check that
\[\sum_{k=1}^{a_{n}} \frac{L(k)}{k}   \quad \mathop{\sim}_{n \rightarrow \infty} \quad  2  \sum_{k=1}^{a_{n}/2} \frac{L(k)}{2k}.\]
But this simply follows from \cite[Proposition 1.5.9 a]{BGT89}, which ensures that the function $n \mapsto \sum_{k=1}^{n} \frac{L(k)}{k}$ is slowly varying at infinity. The proof is now complete.
 \end{proof}

\subsection{Convergence to a multiple of the CRT}
\label{ssec:finalproof}
We are finally in position to establish \cref{thm:CRT}. Before that, let us recall a basic fact about the Gromov--Hausdorff topology (see \cite[Theorem 7.3.25]{burago_course_2001}). If $(E_1,d_1)$ and $(E_2,d_2)$ are metric spaces, a \textit{correspondence} between $E_1$ and $E_2$ is a subset $\mathcal{R}$ of $E_1\times E_2$ such that for every $x_1 \in E_1$, there exists $x_2\in E_2$ such that $(x_1,x_2) \in \mathcal{R}$, and conversely. Now, the \textit{distorsion} $\textup{dis}(\mathcal{R})$ of the correspondence $\mathcal{R}$ is defined by
\[\textup{dis}(\mathcal{R})\coloneqq\sup\{|d_1(x_1,y_1)-d_2(x_2,y_2)| : (x_1,x_2), (y_1,y_2) \in \mathcal{R}\}.\] Then, the Gromov--Hausdorff distance between $(E_1,d_1)$ and $(E_2,d_2)$ is given by
\[d_{\textup{GH}}((E_1,d_1), (E_2,d_2))= \frac{1}{2} \inf_{\mathcal{R}}{\textup{dis}(\mathcal{R})},\] where the supremum is over all correspondences between $(E_1,d_1)$ and $(E_2,d_2)$.

This section involves Aldous' Brownian CRT \cite{Ald93}, whose construction we now recall.  Recall that $\mathbbm{e}$ is the normalized Brownian excursion and introduce a pseudo-distance on $[0,1]$ by setting
\[d_\mathbbm{e}(s,t)=\mathbbm{e}_s+\mathbbm{e}_t-2\min_{s\wedge t \leq u \leq s\vee t}\mathbbm{e}_u, \quad s,t \in[0,1].\] We also let
\[s \approx t \quad \text{if and only if} \quad d_\mathbbm{e}(s,t)=0, \quad s,t \in [0,1].\] Then, the CRT is the quotient space $\mathcal{T}_{ \mathbbm{e}}\coloneqq[0,1] _{/ \approx}$, equipped with the distance $d_\mathbbm{e}$.

\begin{proof}[Proof of \cref{thm:CRT}] First of all, by Skorokhod's representation theorem, we may assume that the convergence of Proposition \ref{prop:cvjointe} holds almost surely, and we aim at proving that the convergence of Theorem \ref{thm:CRT} also holds in this sense.  Recall that $u_0, \ldots, u_{n-1}$ are the vertices of $\Tc_n$ listed in lexicographical order. We  let $\mathbf{p}_\mathbbm{e}$ be the canonical projection from $[0,1]$ onto $\Tc_\mathbbm{e}$, and build a correspondence $\mathcal{R}_n$ between $\frac{1}{B_{n}}\cdot \Loop(\Tn)$ and $ c_{\mu}\cdot \sqrt{2}\Tc_{\mathbbm{e}}$ as follows:
\[\mathcal{R}_n\coloneqq\{(u_i,\mathbf{p}_\mathbbm{e}(s))\in \Loop(\Tn) \times \Tc_\mathbbm{e} : \ i=\lfloor (n-1)s\rfloor, \ s\in[0,1], \ 0\leq i < n\}.\] Let us show that the distorsion of $\mathcal{R}_n$ vanishes as $n\rightarrow \infty$. We argue by contradiction and assume that there exists $\varepsilon>0$, sequences $(i_n : n\geq 1)$ and $(j_n : n\geq 1)$ with $i_n,j_n\in\{0,\ldots, n-1\}$ and $(s_n : n\geq 1)$ and $(t_n : n\geq 1)$ with $s_n,t_n \in [0,1]$ such that for every $n\geq 1$, $(u_{i_n},\mathbf{p}_\mathbbm{e}(s_n))\in \mathcal{R}_n$, $(u_{j_n},\mathbf{p}_\mathbbm{e}(t_n))\in \mathcal{R}_n$ and
\[\left| \frac{1}{B_{n}}\dc_{\Tn}(u_{i_n},u_{j_n}) -  c_{\mu}  \sqrt{2} d_{\mathbbm{e}}(s_n,t_n) \right|>\varepsilon.\]  By compactness, up to extraction, we may assume that $\frac{i_n}{n}\rightarrow s\in [0,1]$ and $\frac{j_n}{n}\rightarrow t\in [0,1]$ as $n\rightarrow \infty$, so that $s_n \rightarrow s$ and $t_n \rightarrow t$ as well. Without loss of generality, we can also assume that $s \leq t$. By the construction of $\Loop(\Tn)$, following \cite[Equation (4.3)]{CK14b} we have for every $n\geq 1$
\begin{equation}\label{eqn:DistLoop}
	\left| \dc_{\Tn}(u_{i_n},u_{j_n}) - \left(\Hc_{i_n}(\Tn)+\Hc_{j_n}(\Tn)-2\Hc_{\textsf{m}(i_{n},j_{n})}(\Tn) \right) \right| \leq k_{u_{\textsf{m}(i_{n},j_{n})}}(\Tc_n),
\end{equation}where $\textsf{m}(i_{n},j_{n})$ is the index of the most recent common ancestor of $u_{i_n}$ and $u_{j_n}$ in the lexicographical order of $\Tc_n$.

Now, recall that $x_0, \ldots, x_{2(n-1)}$ are the vertices of $\Tc_n$ listed in contour order. We follow the guidelines of \cite[Section 2]{Du03} to compare the lexicographical and contour orders of vertices in $\Tc_n$. We set
\[b_n(i)\coloneqq2i-\H_i(\Tn), \quad 0 \leq i <n,\] so that $b_n(i)$ is the index of the first visit of the vertex $u_i$ in the contour order of $\Tc_n$. As a consequence, we have
\begin{equation}\label{eqn:ContourAndLexico}
	u_i = x_{b_n(i)}, \quad 0 \leq i <n.
\end{equation} Moreover, since $\tfrac{B_{n}}{n}\H(\Tn)$ converges in virtue of \cref{prop:cvjointe}, the convergence \begin{equation}\label{eqn:GluingContourLexico}
	\sup_{0 \leq t \leq 1}{\left| \frac{1}{2n}b_n(\lfloor t(n-1) \rfloor) -t \right|}= \sup_{0 \leq t \leq 1}{\left| \frac{1}{2n}\left(2\lfloor t(n-1) \rfloor - \H_{2\lfloor t(n-1) \rfloor}(\Tn)\right) -t \right|} \underset{n \rightarrow \infty}{\longrightarrow} 0,
\end{equation}
holds almost surely. But the quantity $\Big| \tfrac{1}{B_{n}}\Hc_{\textsf{m}(i_{n},j_{n})}(\Tn) -  c_{\mu} \inf_{s\leq u \leq t}\sqrt{2} \mathbbm{e}_{u} \Big|$ is bounded from above by
\[ \sup_{0 \leq i<n}\Big|  \frac{1}{B_{n}} \Hc_{i}(\Tn) - c_{\mu} \frac{B_{n}}{ n }\H_{i}(\Tn) \Big| +  c_{\mu}\Big|  \frac{B_{n}}{n}\H_{\textsf{m}(i_{n},j_{n})}(\Tn) -   \inf_{s\leq u \leq t} \sqrt{2}\mathbbm{e}_{u} \Big|.\]
Since $u_{\textsf{m}(i_{n},j_{n})}$ is the most recent common ancestor between $u_{i_n}$ and $u_{j_n}$ in $\Tc_n$, using \eqref{eqn:ContourAndLexico}, we get 
\[\H_{\textsf{m}(i_{n},j_{n})}(\Tn)=\C_{b_n(\textsf{m}(i_{n},j_{n}))}(\Tn)=\inf_{b_n(i_n) \leq k \leq b_n(j_n)} \C_k(\Tn)=  \inf_{ \frac{b_{n}(i_{n})}{2n} \leq t \leq \frac{b_{n}(j_{n})}{2n}} \C_{2nt}(\Tn).\] From Proposition \ref{prop:cvjointe} and \eqref{eqn:GluingContourLexico} we deduce that 
\[\frac{1}{B_{n}} \Hc_{\textsf{m}(i_{n},j_{n})}(\Tn) \quad \xrightarrow[n \rightarrow \infty]{} \quad   c_{\mu} \cdot \inf_{s\leq u \leq t} \sqrt{2} \mathbbm{e}_{u}\quad \text{almost surely}.\] By Equation \eqref{eqn:DegreMax}, we get by passing to the limit into \eqref{eqn:DistLoop} that 
\[\frac{1}{B_{n}}\dc_{\Tn}(u_{i_n},u_{j_n}) \quad \xrightarrow[n \rightarrow \infty]{} \quad   c_{\mu} \cdot  \sqrt{2} d_{\mathbbm{e}}(s,t) \quad \text{almost surely},\] thus a contradiction, and the proof is complete. \end{proof}

\section{Applications to random planar maps}\label{sec:Maps}

\subsection{A modified looptree} In view of our applications to random planar maps, we need to extend Theorems \ref{thm:circle} and \ref{thm:CRT} to a modified version of the looptree $\Loop(\tau)$, that was first introduced in \cite{CK15} and whose definition we now recall. 

With every plane tree $ \tau$, we associate a planar map $\Loopb( \tau)$, that is obtained from $\Loop(\tau)$ by contracting the edges $(u,v)$ such that $v$ is the last child of $ u$ in lexicographical order in $\tau$ (see Figure \ref{fig:loopb} for an example). We still view $ \Loopb( \tau)$ as a compact metric space by endowing its vertices with the graph distance.  
 \begin{figure}[!h]
 \begin{center}
 \includegraphics[width=  \linewidth]{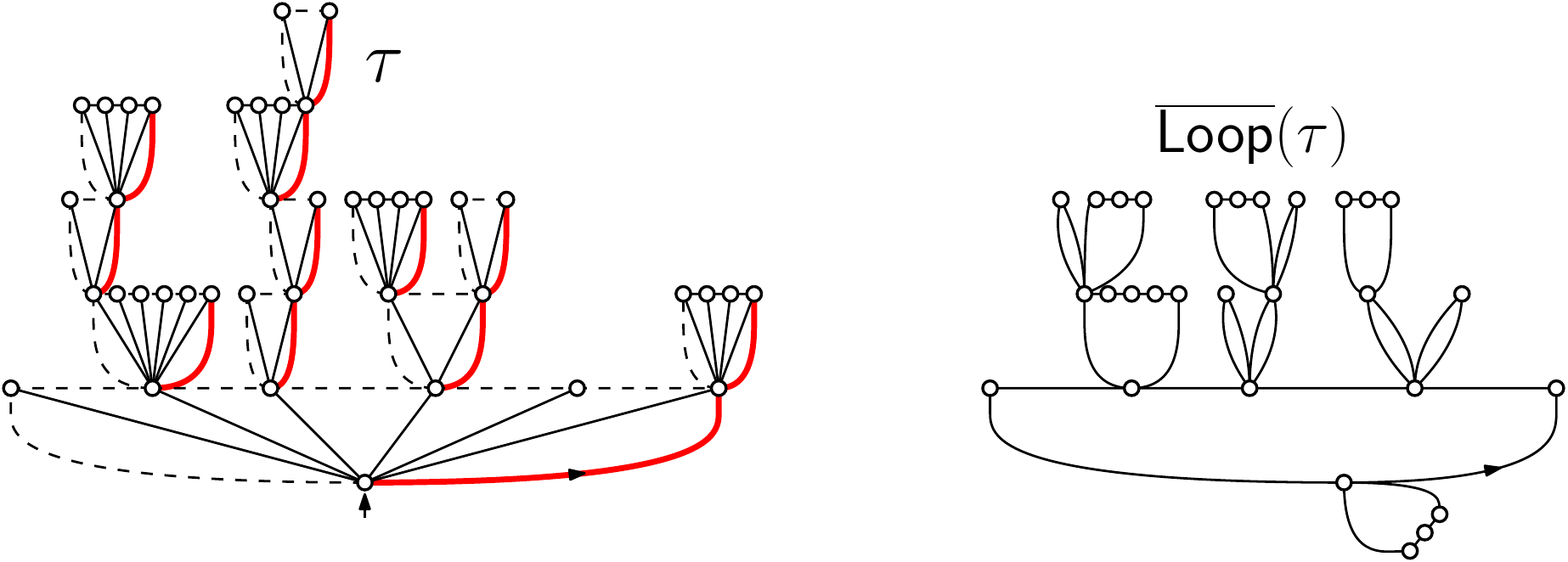}
 \caption{ \label{fig:loopb}A discrete tree $\tau$ (with $\Loop(\tau)$ in dashed edges) and its associated planar map $  \Loopb( \tau)$. The contracted edges are in bold red.}
 \end{center}
 \end{figure}
 
It is a simple matter to check that a result similar to Theorem \ref{thm:circle}  holds with $\Loop$ replaced by $\Loopb$, with almost the same proof. However, for Theorem \ref{thm:CRT}, distances are changed by a constant factor when replacing $\Loop$ by $\Loopb$. The proof of the following theorem goes along the same lines as that of Theorem \ref{thm:CRT}, and we leave details to the reader.

\begin{theorem} \label{thm:CRT2}
 Let $\mu$ be a critical offspring distribution with finite positive variance $\sigma_\mu^{2}$. For every $ n \geq 1$, let $ \Tc_{n}$ be a $ \BGW_{ \mu}$ tree conditioned on having $n$ vertices. Then the convergence
  \begin{eqnarray*}  \frac{1}{\sqrt{n}}\cdot \Loopb ( \Tc_{n}) & \xrightarrow[n\to\infty]{} &    { \frac{2}{ \sigma_\mu}}  \cdot   \frac{1}{4}\left( \sigma_\mu^2+ \mu(2 \Z+) \right) \cdot  \mathcal{T}_{ \mathbbm{e}} \end{eqnarray*}
  holds in distribution for the Gromov--Hausdorff topology.
  \end {theorem}

\subsection{Applications to random planar maps}

\subsubsection{Maps}
\label{sec:ssmaps}

Recall that a \textit{planar map} is a proper embedding of a finite connected graph in the sphere $\mathbb{S}^2$, viewed up to orientation-preserving homeomorphisms. The faces are the connected components of the complement of the embedding, and the degree $\deg(f)$ of the face $f$ is the number of oriented edges incident to this face. We systematically consider \textit{rooted} maps, i.e., with a distinguished oriented edge called the \textit{root edge}. The face $f_*$ on the right of the root edge is the \textit{root face}. We focus on planar maps \textit{with a boundary}, meaning that the root face is an \textit{external face} whose incident edges and vertices form the \textit{boundary} of the map. The boundary of a map $\m$ is denoted by $\partial\m$ and the degree $\#\partial\m$ of the external face is called the \textit{perimeter} of~$\m$. For technical reasons, it is sometimes simpler to deal with the \textit{scooped-out} map $\Scoop(\m)$, which is obtained from $\partial\m$ by duplicating the edges whose both sides belong to the root face (see Figure \ref{fig:scoop}). Note that $\partial \m$ and $\Scoop(\m)$ define the same metric space.

We also restrict ourselves to \textit{bipartite} maps, in which all the faces have even degree. The set of bipartite maps is denoted by $\mathscr{M}$, and the set of bipartite maps with perimeter $2k$ by $\mathscr{M}_{k}$. By convention, the map $\dagger$ made of a single vertex is the only element of $\mathscr{M}_0$. 

\begin{figure}[!h]
 \begin{center}
 \includegraphics[width=\linewidth]{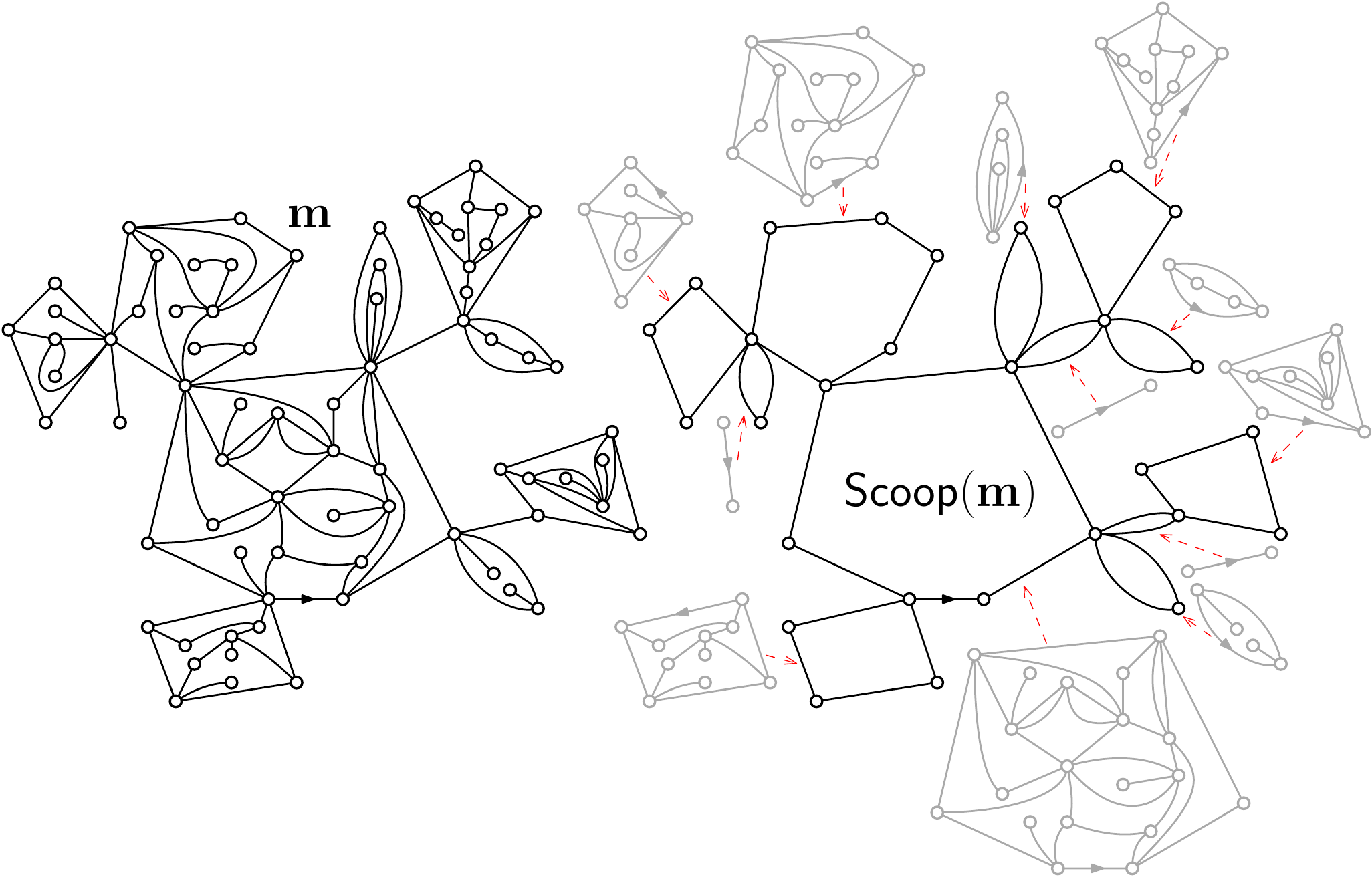}
 \caption{ \label{fig:scoop}A rooted planar map $\m$ and its scooped-out map $  \Scoop( \m)$.}
 \end{center}
 \end{figure}

\subsubsection{Boltzmann distributions} Let us recall the construction of the Boltzmann distributions on bipartite maps and their main properties. We first fix a \textit{weight sequence} $\q=(q_k : k\geq 1)$ of nonnegative real numbers, and define the \textit{Boltzmann weight} of a bipartite map $\m$ by
\begin{equation*}
	w_\q(\m)\coloneqq\prod_{f \in \textup{Faces}(\m)}q_{\deg(f)/2},
\end{equation*} with the convention $w_\q(\dagger)=1$. We will also use the following function introduced in~\cite{MM07}:
\[
	f_\q(x) \coloneqq \sum_{k=1}^{\infty}\binom{2k-1}{k-1}q_k x^{k-1}, \quad x\geq 0.
\] We say that $\q$ is admissible if the equation
\[
	f_\q(x)=1-\frac{1}{x}, \quad x>0
\] has a solution, and the smallest solution is denoted by $Z_\q$. Then, we have $w_\q(\mathscr{M})<\infty$ and the \textit{Boltzmann distribution with weight sequence} $\q$ is the probability measure defined by
\[\Pq(\m)=\frac{w_\q(\m)}{w_\q(\mathscr{M})}, \quad \m\in \mathscr{M}.\] 

 The partition function for bipartite maps with a fixed perimeter and the associated generating function
\begin{equation*}
	F_k\coloneqq\frac{1}{q_k}\sum_{\m\in\mathscr{M}_{k}}w_\q(\m), \quad k\geq 0, \qquad \text{and} \qquad F(x)\coloneqq\sum_{k=0}^\infty F_k x^k, \quad x\geq 0,
\end{equation*} will also play a role (here, the factor $1/q_k$ stands for the fact that the root face receives no weight). The radius of convergence of the latter power series is $r_\q:=(4 Z_\q)^{-1}$.

A powerful tool to study Boltzmann distributions is the Bouttier--Di Francesco--Guitter bijection \cite{BDFG04} that associates to every (pointed) map a tree with labels associated to vertices at even height. The tree associated to a (pointed) Boltzmann map by this bijection is then a so-called \textit{two-type} $\BGW$ tree, whose distribution is given in \cite[Proposition 7]{MM07}. 

The weight sequences $\q$ can then be classified throughout the distribution of this tree, following \cite{MM07,LGM09,BBG12}. This classification can be rephrased as follows: we say that $\q$ is critical if the expected number of vertices of the tree (or, equivalently, of the associated Boltzmann map) is infinite, and subcritical otherwise. Moreover, we say that $\q$ is generic if the offspring distribution of vertices at odd height in the tree (which one can think of as the law of the degrees of the faces in the map) has finite variance, and $\q$ is non-generic with parameter $\alpha \in (1,2)$ if this offspring distribution falls in the domain of attraction of a stable law with parameter $\alpha$. As we mentioned in the Introduction, non-generic critical sequences with parameter $\alpha \in (1,3/2)$ are often called \textit{dense}, while for $\alpha \in (3/2,2)$ they are called \textit{dilute}.

\subsubsection{Proof of Corollaries \ref{cor:ScalingDilute} and \ref{cor:ScalingSubcritical}}

The following result is a direct consequence of \cite[Corollaries 3.4 and 3.7]{Ric17}, combined with \cite[Lemma 4.3]{CK15}.

\begin{lemma}
	Let $\q$ be an admissible weight sequence. For every $k\geq 0$, let $\Mc_{k}$ (resp.\  $\Mc_{\geq k}$) be a Boltzmann map with weight sequence $\q$ conditioned to have perimeter $2k$ (resp.\ at least $2k$). Then, there exists an offspring distribution $\nu$ such that the following identities hold in distribution	\[\Scoop(\Mc_k)  \quad = \quad  \Loopb(\Tc_{2k+1}) \qquad \text{and} \qquad \Scoop(\Mc_{\geq k}) \quad  =   \quad \Loopb(\Tc_{\geq 2k+1}),\] where $\Tc_{2k+1}$ (resp.\ $\Tc_{\geq 2k+1}$) is a $\BGW_\nu$ tree conditioned to have $2k+1$ vertices (resp.\ at least $2k+1$ vertices). 
\end{lemma}

The offspring distribution  $\nu$ is given explicitly in \cite[Corollary 3.4]{Ric17}, and we also have the following information.
\begin{itemize}
		\item[--] If $\q$ is subcritical (case $a=3/2$ in \cite{Ric17}), then $\nu$  has mean $m_\nu=1$ and finite variance \[\sigma^2_\nu=\left(\frac{F(r_\q)}{1-Z_\q^2f_\q'(Z_\q)}\right)^2,\]
		see \cite[Lemma 3.5 and (41)]{Ric17}.
		\item[--] If $\q$ is generic critical (case $a=5/2$ in \cite{Ric17}) or dilute (case $a \in (3/2,2)$ in \cite{Ric17}) , then $\nu$ has mean
		\[m_\nu=\frac{1}{1+\frac{F(r_\q)}{2r_\q F'(r_\q)}}<1,\] and falls in the domain of attraction of a stable distribution with parameter parameter $3/2$ (in the generic critical regime) or $\alpha-1/2 \in (1,3/2)$ (in the dilute regime with parameter $\alpha \in (3/2,2)$).
\end{itemize}

Corollaries \ref{cor:ScalingDilute} and \ref{cor:ScalingSubcritical} then immediately follow, with
\[ K_\q= { \frac{2}{ \sigma_\nu}}  \cdot   \frac{1}{4}\left( \sigma_\nu^2+ 1 \right),\]
and
 $J_\q$ a random variable defined by $\Pr{J_{\q} \geq x}=(\tfrac{1-m_\nu}{x})^{\alpha-1/2}$ for $x \geq 1-m_\nu$ (in the dilute regime with parameter $\alpha \in (3/2,2)$) or $\Pr{J_{\q} \geq x}=(\tfrac{1-m_\nu}{x})^{3/2}$ (in the generic critical regime).

\begin{remark}\label{rem}
We can now  discuss more precisely the assumption that the $\BGW$ tree has \textit{at least} $n$ vertices in Theorem \ref{thm:circle}, rather than exactly $n$ vertices. In the above application, the offspring distribution $\nu$ is given in terms of the partition function for maps with a \textit{simple} boundary (that is, with no self-intersections) of fixed perimeter (see \cite[Corollary 3.4]{Ric17}). However, we are not able to obtain an asymptotic behaviour for these quantities (only for the remainder of their sum) as explained in \cite[Remark 2.8]{Ric17}. For this reason, the assumptions of \cite{Kor15} are a priori not satisfied by the offspring distribution $\nu$, which forces us to  use a weaker conditioning (so that the weaker regularity assumption is fulfilled by $\nu$).

Finally, in the non-generic critical regime with parameter $\alpha=3/2$, the probability measure $\nu$ can be either subcritical or critical, and is expected to be in the domain of attraction of a Cauchy distribution. However, the last assertion is only established in \cite[Section 6]{Ric17} for a particular weight sequence $\q$ (and then, $\nu$ is critical, so that Theorem \ref{thm:circle} does not apply).

 \end{remark}

\bibliographystyle{alpha}

\end{document}